\documentclass[10pt,oneside]{amsart}
\usepackage{amssymb, amscd, amsmath, amsthm, latexsym, hyperref,ams refs}
\usepackage{pb-diagram, fancyhdr, graphicx, psfrag, enumerate}
\usepackage[text={6.3in,8.5in},centering,letterpaper,dvips]{geometry}
\usepackage{enumitem}\setlist[enumerate,1]{font=\upshape, itemsep=1ex}\setlist[itemize,1]{font=\upshape, itemsep=1ex}
\usepackage{pinlabel}

\def\Z{{\mathbb Z}}

\def\Q{{\mathbb Q}}

\def\C{{\mathbb C}}
\def\P{{\mathbb P}}
\def\CP{\C P}

\def\co{\colon\thinspace}
\def\calu{\mathcal{U}}

\def\cs{\mathbin{\#}}
\DeclareMathOperator{\cfk}{\rm CFK}
\DeclareMathOperator{\cf}{\rm CF}

\newcommand{\spinc}{\ifmmode{{\mathfrak s}}\else{${\mathfrak s}$\ }\fi}
\newcommand{\spinct}{\ifmmode{{\mathfrak t}}\else{${\mathfrak t}$\ }\fi}
\newcommand{\spincr}{\ifmmode{{\mathfrak r}}\else{${\mathfrak r}$\ }\fi}
\newcommand{\Spc}{Spin$^c$}
\newcommand{\fig}[2] { \includegraphics[scale=#1]{#2} }
\def\U{\Upsilon}
\def\Spinc{Spin$^c$}

\def\ns{\overline{\sigma}_{p,q}(x)}

\newtheorem{theorem}{Theorem}[section]
\newtheorem{lemma}[theorem]{Lemma}
\newtheorem{corollary}[theorem]{Corollary}
\newtheorem{proposition}[theorem]{Proposition}

\theoremstyle{definition}
\newtheorem{definition}[theorem]{Definition}
\newtheorem{example}[theorem]{Example}

\begin{document}
\title{Unknotting with a single twist}
\author{Samantha Allen}
\author{Charles Livingston}
\thanks{This work was supported by a grant from the National Science Foundation, NSF-DMS-1505586.   }
\address{Charles Livingston: Department of Mathematics, Indiana University, Bloomington, IN 47405}\email{livingst@indiana.edu}
\address{Samantha Allen: Department of Mathematics, Dartmouth College, NH 03755}\email{Samantha.G.Allen@dartmouth.edu}


\begin{abstract}  Given a knot $K \subset S^3$, is it possible to unknot it by performing a single twist, and if so, what are the possible linking numbers of such a twist?   We develop obstructions to unknotting using a twist of a specified linking number.  The obstructions we describe are built using classical knot invariants, Casson-Gordon invariants, and Heegaard Floer theory.
\end{abstract}

\maketitle


\section{Introduction}

Figure~\ref{fig:t23} presents three illustrations of the right handed trefoil knot, $T_{2,3}$.  In each, performing a full twist on the parallel strands that pass through the small circle results in an unknot.  In the first two cases the required twist is negative, and in the last it is positive.  The linking numbers of the twists, which by convention are always positive,  are 2, 3, and 0, respectively.  Thus, we say the set of unknotting twist indices,  denoted $\calu$, satisfies  $ \{ 2^-, 3^-, 0^+\} \subset \calu(T_{2,3} )$. The reader is invited to show that for the figure eight knot, $4_1$, $ \{ 2^-,   0^-, 0^+, 2^+\} \subset \calu(4_1) $.    The results of this paper will imply that these two containments are, in fact, equalities.

\begin{figure}[h]
\fig{.6 }{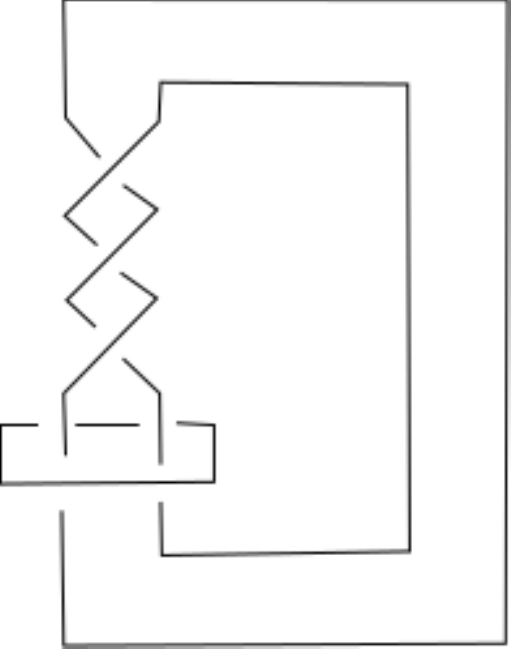} \hskip.7in
\fig{.45}{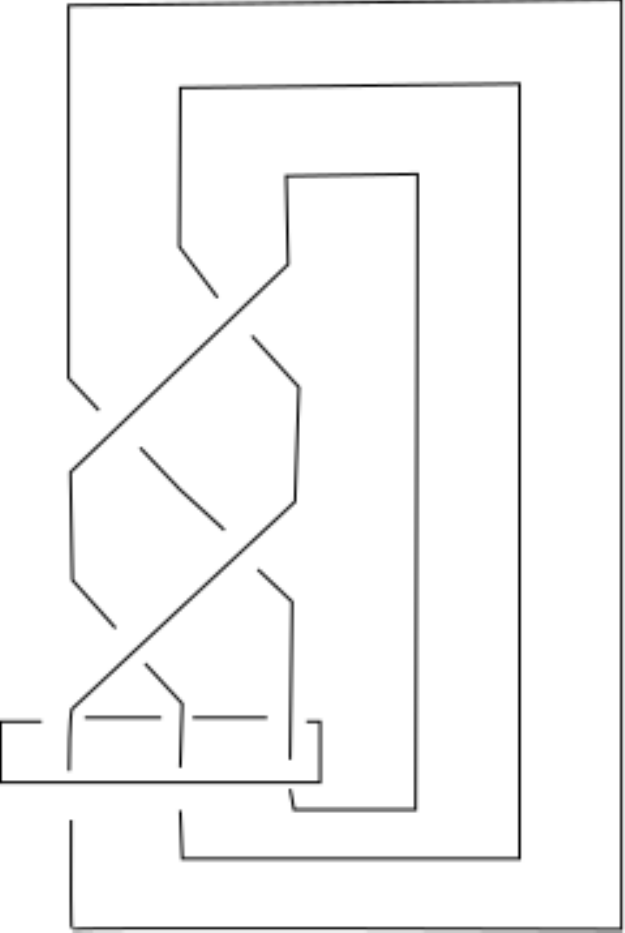} \hskip.7in
\fig{.6}{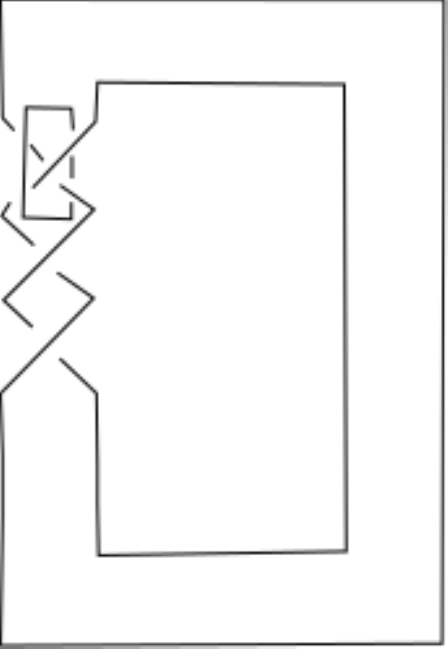} \caption{Unknotting the trefoil}
\label{fig:t23}
\end{figure}

Our  goal is to consider the question of which knots can be unknotted with a single   twist  and, more generally, to describe  tools for analyzing $\calu(K)$.  This problem has been extensively studied, often in the more general setting in which the one operation consists of introducing  perhaps more than one   full twist on the parallel strands.    A sampling of references includes~\cite{MR3619734,MR1355072,MR1194998,MR1355073, MR1282760,MR1177414,MR2233011,MR1871306, MR2016983}.  Of particular note is work of Ince~\cite{MR3513844, MR3685608} which applies Heegaard Floer theory in the case of linking number 0 and that of Sato~\cite{MR3823992}, which considers a related slicing problem in $\C\P^2$.

From the perspective of classical knot theory, a theorem of Ohyama~\cite{MR1297516} heightens the interest in unknotting with one twist:  every knot can be unknotted with   two full twists. In~\cite{livingston2019nullhomologous} it was observed that the linking numbers of the two twists can be any consecutive pair of integers, and that if one requires that the linking numbers be 0, then up to $2g$ twists might be required, where $g$ is the three-genus of the knot.

Some of our results concerning  $\calu(K)$ overlap with previous ones, albeit with alternative and at times simpler proofs.      Other results, especially those based on Heegaard Floer theory, are new, as is our examination of new ways to use the various approaches in conjunction.

The basis of  much of our work is the observation that if $K$ can be unknotted with a single twist, then three-manifolds built by surgery on $K$ bound four-manifolds with special properties.  This in turn lets us apply a range of tools to the problem, including those that arise in classical knot theory, Casson-Gordon theory, and Heegaard Floer theory.  Of special relevance  is the work of Aceto and Golla~\cite{MR3604383}  applying Heegaard Floer theory to the question of which surgeries on a given knot $K$ bound rational homology four-balls.

We now summarize a few of the main results presented in the paper.

\smallskip
\noindent{\bf Outline}
\begin{itemize}

\item Section~\ref{sec:surgery} presents the basic geometric observations that form the basis of our later work.  This includes the observation that if $K$ can be unknotted with a single twist of sign $s$ and linking number $l$, then $S^3_{-sl^2 - s}(K) \cong S^3_{sl^2 +s}(J)$ for some knot $J$.  We also note the well-known fact that $S^3_{sl^2}(K)$ bounds a four-manifold $W$ with $H_1(W) \cong \Z_l$ and, if $l \ne 0$, $H_2(W) = 0$; in addition, we note that in fact $H_1(W) \cong \pi_1(W)$ and that the map $\pi_1(S^3_{sl^2}(K)) \to H_1(W)$ is surjective.

\item  Sections~\ref{sec:homology},~\ref{sec:link},~\ref{sec:cg}, and~\ref{sec:arf}  present results related to  homological invariants associated to branched cyclic covers and the infinite cyclic cover.  In   Section~\ref{sec:homology} the focus is on the ranks of the homology groups.   As we describe, our results that arise from  $\Z[t, t^{-1}]$--coefficients instead of $\Q[t, t^{-1}]$--coefficients  depend on the use of Gr\"obner bases.     In cases in which the rank of the homology is not sufficient to provide necessary obstructions,   we observe in Section~\ref{sec:link} that the $\Q/\Z$--valued linking form can provide stronger obstructions.

Section~\ref{sec:cg} explores the use of the knot signature function, using an approach developed by  Casson and Gordon; our results include a short proof of a theorem of Nouh-Yasuhara~\cite{MR1871306} which they stated for torus knots.   In combination, these results place strong limits of the possibility of the set $\calu(K)$ containing both positive and negative entries; for instance, if $\{4^{-} , 5^+\} \subset \calu(K)$, then  the genus of $K$ is at least $23$; another new application of the signature result is that, with a few exceptions, if $\gcd(l_1, l_2) \ne 1$, then $\{l_1, l_2\} \not\subset \calu(K)$ for any knot $K$ (with any choice of signs).

Finally, Section~\ref{sec:arf} complements the signature results of Section~\ref{sec:cg}  with a discussion of the Arf invariant.  Applications to torus knots are described.  In addition, among classical knot invariants, the Arf invariant is to our knowledge the strongest one that can address the possibility of $1^\pm \in \calu(K)$.  Later we will see that Heegaard Floer obstructions offer alternative obstructions, but even for low crossing alternating knots of fewer than 13 crossing, the Arf invariant provides an obstruction in over 600 cases in which the Heegaard Floer obstructions vanish.

\item Section~\ref{sec:hf} provides  background on Heegaard Floer theory, summarizing the essential properties of the knot invariants $\nu^+(K)$ and $V_i(K)$.  We also describe the needed properties of the three-manifold correction terms, $d(Y, \spinc)$, restricting to the case of $Y = S^3_m(K)$.

\item Section~\ref{sec:hf2} presents an obstruction to unknotting based on the invariants $V_i(K)$.  The obstruction itself was first presented by Sato~\cite{MR3823992} and our  result also follows from work of Aceto-Golla~\cite{MR3604383}.  We include our own proof; in our setting we are able to give very short, and hopefully accessible, arguments.  In addition, we present new applications of these obstructions.  Section~\ref{sec:hf3} presents much stronger constraints on the invariants $V_i(K)$.  These are specific to the unknotting problem and do not apply in the more general settings of~\cite{MR3604383,MR3823992}.

Section~\ref{sec:hf6} presents obstructions based on the Heegaard Floer invariants associated to cyclic covers of a knot $K$, more specifically, $d(M_2(K), \spinc)$.  In Section~\ref{sec:hf5} we present obstructions based on the Upsilon invariant of Ozs\'ath-Stipsicz-Szab\'o~\cite{MR3667589}. As we will make clear, the Upsilon invariant is theoretically  no stronger than the $V_i$--invariants, but it has  the advantage of being much more computable; this is illustrated with an example of  connected sums  of torus knots.

\item Section~\ref{sec:alt} discusses the case of alternating knots, in which computations are most accessible.

\item Section~\ref{sec:questions} presents some comments and open problems concerning sets $\calu(K)$.

\item In  the appendix we present a few technical results and present a summary of an analysis of prime knots of eight or fewer crossings.

\end{itemize}
\smallskip

\smallskip

\noindent{\it Acknowledgement}\ \  Mohamed Ait Nouh provided us with many pointers regarding the general problem of untwisting;  this paper  significantly benefitted  from his feedback.  Marco Golla and Paolo Aceto also gave us insightful feedback which was of great help in improving the exposition.

\section{Geometric results related to  unknotting twists}\label{sec:surgery}

Throughout this paper we will use surgery descriptions of  three-manifolds,   knots, and their branched covering spaces.  A basic reference  is the text by Rolfsen~\cite[Chapter 9H]{MR0515288}.   More details can be found in~\cite{MR1707327} and original sources such as~\cite{MR0467753}.

\subsection{Three-dimensional aspects of surgery diagrams and unknotting}

Let $K \subset S^3$ be knot.  We denote by $S^3_n(K)$ the three-manifold formed by $n$--surgery on $K$.  If $(K, J)$ is link, we write $S^3_{n,m}(K,J)$ for the three-manifold formed by performing $n$-- and $m$--surgery on $K$ and $J$, respectively.

In the case of the unknot $U$ and    $s = \pm 1$ we have  $S^3_s(U) \cong S^3$.  It follows that $S^3_{n,s}(K,U) \cong  S^3_{n'}(K')$ for some $n'$ and $K'$.    As described in~\cite[Chapter 9G]{MR0515288}, $K'$ is the  knot formed by performing a full twist to the strands of $K$ passing through $U$, twisting left or right depending on whether $s = 1$ or $s = -1$, respectively.  The new surgery coefficient is  $n' = n-sl^2$, where $l = \text{link}(K,U)$.  (In general, the linking number is defined for {\it oriented} links.  Here, we chose orientations so that the linking number is nonnegative.)    In summary, we have the next result.

\begin{theorem}  Suppose that $K$ can be unknotted with a single twist of sign $s = \pm 1$ and linking number $l$.  Then for all $n$, 
\[S^3_n(K) \cong S^3_{n + sl^2, s}(U_1, U_2),\]
where $U_1$ and $U_2$ are both unknotted.
\end{theorem}

\begin{corollary}In the setting of the theorem:
\begin{enumerate}

\item  $S^3_{-sl^2}(K) \cong S^3_{0, s}(U_1, U_2)$, where $\text{link}(U_1, U_2) = l$.

\item $S^3_{-sl^2 -s}(K) \cong S^3_{-s, s}(U_1, U_2) \cong S^3_{sl^2 + s}(J)$, for some knot $J$.

\end{enumerate}
\end{corollary}

\subsection{Four-dimensional aspects of surgery diagrams and unknotting.}  Recall that 
\[S^3_0(U) \cong S^1 \times S^2 \cong \partial (S^1 \times B^3).\]  
Thus, $S^3_{0, s}(U_1, U_2) \cong \partial W$, where $W$ is built from $ S^1 \times B^3$ by adding a single two-handle.  The next theorem then follows readily; we write $\Z_l $ for $\Z / l\Z$, so in the special case $l = 0$ we have $\Z_l \cong \Z$.

\begin{theorem} Suppose that $K$ can be unknotted with a single twist of sign $s$ and linking number $l$.  Then 
\begin{enumerate}

\item  $S^3_{-sl^2}(K) =  \partial W$ where $W $ is built from $ S^1 \times B^3$ by adding a single two-handle, added with framing $s$ in the standard surgery diagram for $S^1 \times B^3$.

\item  The attaching curve for the two-handle represents $l \in H_1(S^1 \times B^3)$, so $\pi_1(W) \cong H_1(W) \cong \Z_l$.

\item  The map induced by inclusion, $H_1(S^3_{-sl^2}(K)) \to H_1(W)$ corresponds to  the surjection $\Z_{l^2} \to \Z_l$.

\end{enumerate}
\end{theorem}


\section{Single twist unknotting: homological constraints}\label{sec:homology}

The simplest obstructions to unknotting with a single twist arise from homological properties of the  cyclic branched covers and the infinite cyclic cover of the knot.  To describe these, we let   $M_q(K)$ denote the $q$--fold cyclic   branched cover of $S^3$ with branching set  $K$  and let $M_\infty(K)$ denote the infinite cyclic cover.  We begin with a definition.

\begin{definition}  A triple $(U_1, U_2, s)$ where $(U_1 , U_2)$ is a link with unknotted components and $s= \pm 1$ is called a {\it surgery diagram} for a knot $K$ if $U_1$ represents the knot $K$ in   $S^3_{s}(U_2) \cong S^3$.

\end{definition}

\begin{theorem} \label{thm:branched}  If $K$ can be unknotted with a twist of linking number $l$ and   $l$ is divisible by $q$, then
$H_1(M_q(K),\Z)$ is generated by  $q$ elements.   If $l=0$, then the  homology group $H_1(M_\infty(K),\Z)$ is generated by a single element as a $\Z[t,t^{-1}]$--module.

\end{theorem}
\begin{proof}
Details of the construction of branched covers of knots from surgery diagrams are presented in~\cite[Chapter 6C] {MR0515288}.  
Starting with the surgery diagram for $K$, $(U_1, U_2, s)$,   one can construct a surgery diagram of $M_q(K)$; the surgery link consists of the components of the preimage of  $U_2$ in the $q$--fold branched cover of $S^3$ over $U_1$, $M_q(U_1) \cong S^3$.   There are $q$ components.  In particular, the homology of $M_q(K)$ has a presentation with $q$ generators.

The case of $l=0$ is  similar; in brief, the infinite cyclic cover is given by surgery on the set of translates of a single curve in the infinite cyclic cover of the unknot.  Rolfsen's illustration of infinite cyclic covers~\cite[Chapter 7C] {MR0515288} makes the result transparent.
\end{proof}

\noindent{\bf Note.}  Recall that $\text{rank}(H_1(M_q(K),\Z)) \le 2g(K)$.  This follows from a theorem of Seifert~\cite{MR1512955}; see also~\cite[Chapter 8, D9]{MR0515288} and~\cite{MR1201199}.  Thus, the obstruction arising from Theorem~\ref{thm:branched} can provide information only in the case $q \le 2g$.\vskip.05in

\begin{example}
For low-crossing prime knots, this result is of limited value.  Among prime knots of 12 or fewer crossings, for only seven is $\text{rank}(H_1(M_2(K),\Z)) >2$.  These are $12a_{554}, 12a_{750},$ $12n_{553},$ $12n_{554},$ $12n_{555},$ $12n_{556}$, and $12n_{642}$.  Thus, only these seven are obstructed from being unknotted with a single twist with even linking number using $2$--fold branched covers.

For connected sums, the theorem offers stronger results.  For instance, for the trefoil knot, $T_{2,3}$, $H_1(M_2(T_{2,3})) \cong \Z_3$ and thus $3T_{2,3}$ cannot be unknotted with a single twist of even linking number.  In this context, it is worth noting that there are examples of composite knots that can be unknotted with a single twist,~\cite{MR1282760,MR1189955,MR1194998},  and an open conjecture is that all such examples have exactly two prime components.

\end{example}

\subsection{Alexander ideals and Gr\"obner bases}
The homology of the infinite cyclic cover of a knot has a presentation as a $\Z[t,t^{-1}]$--module of the form $A = V -t V^{\sf T}$, where $V$ is a square Seifert matrix of size $2g$.     For $0\le k <2g$, the $k$--elementary ideal (or Alexander ideal) $E_k(K)$ is defined to be the ideal in   $\Z[t,t^{-1}]$ that is generated by the $(2g - k) \times (2g-k)$ minors of $A$.  These ideals are independent of the choice of Seifert matrix $V$ and are invariants of the underlying module.   In general $E_0(K)$ is principal, generated by the Alexander polynomial.  If $H_1(M_\infty(K))$ is generated by a single element, then $E_1(K) = \left<1\right>$.

A reduced Gr\"obner basis of a multi-variable polynomial ideal is a generating set of a specific form.  A basic reference is~\cite{MR2286236}. For us, the relevant properties are that such bases are readily computable by computer packages (we use Wolfram Mathematica~\cite{mathematica}) and permit one to determine whether two given ideals are equal.   To apply Gr\"obner bases to our work, we note that there is a surjection
\[ 
\Z[t,s] \to \Z[t,s]/\left< 1 - ts \right> \cong \Z[t,t^{-1}],
\]
and thus ideals in  $\Z[t,t^{-1}] $ can be analyzed via their preimages in the polynomial ring  $\Z[t,s]$.

\begin{example} In considering the rank of $H_1(M_\infty (K))$  we can use  rational coefficients, in which case the obstruction is more easily computed, or we can work with integer coefficients, in which case  the computation is more complicated but the results are much stronger.   For instance, there are 84 prime knots of 9 or fewer crossings.   Of those, only two, $8_{18}$ and $9_{40}$, have infinite cyclic cover with  noncyclic homology using  $\Q[t,t^{-1}]$--coefficients.  If one switches to $\Z[t,t^{-1}]$--coefficients, an  additional seven  knots are obstructed from being unknotted with a twist of linking number 0 ($9_{35}, 9_{37}, 9_{41}, 9_{46}, 9_{47}, 9_{48}, 9_{49}$).   Among these examples is  $9_{46}$; see~\cite[Chapter 8C] {MR0515288}, where showing that this knot does not have cyclic Alexander module is presented as an exercise.  From what is developed there, it is easily seen that the second Alexander ideal is $\left<  3, 1-t \right> \subset \Z[t, t^{-1}]$.  For the rest of the examples, we used Mathematica to find the  Gr\"obner basis for $E_1(K)$, in each case showing that the module is nontrivial. 
\end{example}

\section{Single twist unknotting: Linking form constraints}\label{sec:link}

In the case that the modules $H_1(M_q(K))$ do  not obstruct an unknotting twist, the linking form on  $H_1(M_2(K))$ can offer much  stronger constraints.   Recall that for any three-manifold $M$ with $H_1(M, \Q) = 0$, there is a  linking form $\text{lk}\co\! H_1(M) \times H_1(M) \to \Q/\Z$.  In the case that  $M = \partial W$, where $H_1(W)= 0$ and the   intersection form of $W$ is presented by a matrix $Q$, the matrix  $Q$ is also a presentation matrix for $H_1(M)$ and $Q^{-1}$ presents the linking form of $M$ with respect to a corresponding generating set of $H_1(M)$.   (Our sign convention is chosen  so that if $M = S^3_n(K)$, then the meridian to $K$ has self-linking $1/n \in \Q/\Z$.)  More generally, if $M$ is given as surgery on a link, then the linking form with respect to the meridians is presented by the inverse of associated surgery  matrix, formed from the linking matrix by using the surgery coefficients as the diagonal entries.

\begin{theorem} If $K$ can be unknotted with a single twist of linking number $2k$ and sign $s$, then the two-fold branched cover $M_2(K)$ is given by surgery on a two-component link with surgery matrix 
\[
Q =   \begin{pmatrix} a & b \\ b & a \end{pmatrix}. 
\]
where $\big| a^2 - b^2\big| = \det (K)$ and $a + k \equiv 1 \mod 2$.
\end{theorem}
\begin{proof} The statement that the surgery matrix is $2 \times 2$ with the diagonal entries equal follows from the discussion of Section~\ref{sec:homology}.   The determinant of a knot is the order of the homology of the 2--fold branched cover, giving the condition that $\big| a^2 - b^2\big| = \det (K)$.

A theorem of Nagami~\cite{MR1772842} states that for a closed four-manifold $W$ with $H_1(W, \Z_2) = 0$, the two-fold branched cover over a surface that represents $2x$ for some homology class $x \in H_2(W)$ is Spin if and only if the mod 2 reduction of  $x$ is dual to the second Stiefel-Whitney class of $W$.  The bounding manifold we have constructed is  the 2--fold branched cover of a punctured $\pm \C\P^2$; because the boundary is $S^3$,   Nagami's result applies: to move to the setting of closed manifolds simply cap off the punctured manifold with a four-ball and the surface with an orientable surface.
\end{proof}

Kauffman and Taylor~\cite{MR388373} proved that if  a knot $K$ bounds a surface $F$ in a four-manifold $W$ and the 2--fold branched cover of $(S^3,K)$ extends over $(W,F)$, then the signature of $K$ is determined by invariants of $W$, the normal bundle to $F$, and the two-fold branched cover of $W$ over $F$.  Restricting to our setting, we consider  a knot  $K$ that can be unknotted with   a single twist of linking number $l$ and sign $s=\pm 1$.  Such a knot  bounds a disk in $s\CP^2\setminus B^4$ with Euler class satisfying $\chi^2 = sl^2 $.   A restatement  of \cite[Theorem 3.1]{MR388373} in this special case   immediately yields the following result.
\begin{theorem}
If a knot $K\subset S^3$ can be unknotted with a single twist of even linking number $l$ and sign $s=\pm1$, then 
\[ \sigma(K) = \sigma(N)-2s+\frac{1}{2}sl^2 \]
where $N$ is the two-fold branched cover of $s\CP^2\setminus B^4$ branched over a disk $\Delta$ such that $\partial  \Delta = K$.
\end{theorem}
This has the following  corollary.
\begin{corollary} \label{cor:two-fold}
If a knot $K\subset S^3$ can be unknotted with a single twist of even linking number $l = 2k$ and sign $s=\pm1$, then the two-fold branched cover of $S^3$ branched over $K$ bounds a four-manifold $N$ with second Betti number $b_2(N) = 2$, signature 
\[ \sigma(N) = \sigma(K)+2s(1- k^2), \]
and intersection pairing with matrix of the form  
\[
Q = \begin{pmatrix} a & b \\ b & a \end{pmatrix}.
\]
The value of $a$ is even or odd depending on whether $k$ is odd or even, respectively.  The value of $|a^2 - b^2| = \det(K)$ and $Q$ is negative definite, indefinite, or positive definite depending on whether $  \sigma(K)+2s-\frac{1}{2}sl^2$ is $-2, 0$ or $2$, respectively.

\end{corollary}

\begin{proof}  The two-fold branched cover $N$ of $s\C\P^2$ branched over the slice disk for $K$ is built from the four-ball be adding two two-handles.  Thus, $H_2(N) \cong \Z^2$.  A rank two form is negative definite if and only if it has signature $-2$; in our case the signature is given by  $  \sigma(K)+2s-\frac{1}{2}sl^2$.  The argument is similar in the positive definite case. 
\end{proof}

\begin{example}
Consider the knot $K =-7_7$.  This is a two-bridge knot $B(21,13)$ and $M_2(K) = L(21,13)$. It satisfies  $\sigma(K) = 0$ and  $\det(K) = 21$.  This knot has unknotting number 1, and a quick examination of its diagram shows that it can be unknotted with a left-handed twists, so  $0^- \in \calu(K)$.  We will show that $0^+\not\in \mathcal{U}(K)$.    Suppose that $K$ could be unknotted with a single positive twist of linking number $0$.  Then Corollary \ref{cor:two-fold} implies that $M_2(K)$ bounds a positive-definite  four-manifold $N$ with  $b_2(N) = 2$ and intersection pairing 
\[
Q= \begin{pmatrix} a & b \\ b & a \end{pmatrix} \] with determinant 21. Up to change of basis, there are only two such matrices:
\[ Q_1 = \begin{pmatrix} 11 & 10 \\ 10 & 11 \end{pmatrix} \text{ and }  Q_2 = \begin{pmatrix} 5 & 2 \\ 2 & 5 \end{pmatrix}. 
\]
This would imply that $H_1(M_2( K)) \cong \Z_{21}$ is generated by an element with self-linking either $11/21$ or $5/21$.  The set of all  self-linking numbers of generators would be given by the set of residues   $11(i^2) \mod 21$ or $5(i^2) \mod 21$, where $\gcd(i, 21) = 1$.  These two sets are $\{2, 8, 11\} $ and $\{5, 17, 20\}$.  On the other hand, as an oriented manifold,   $\Sigma(K) = L(21, 13)$  and the set of  self-linking numbers of generators is given by $\{10, 13, 19\}$.
\end{example}

\section{Single twist unknotting: Casson-Gordon invariants and signatures}\label{sec:cg}

We begin by reviewing Casson-Gordon invariants, restricting to the generality needed for our applications.  Suppose that $M^3$ is a closed oriented three-manifold, $l >0$, and $\phi \co H_1(M) \to \Z_l$ is a homomorphism.  Suppose further that $\phi$ extends to a map $\overline{\phi} \co H_1(W \setminus F) \to \Z_l$, where $W$ is an oriented four-manifold with $\partial W = M$ and $F$ is an embedded, possibly empty, surface.  Then we have the definition 
\begin{equation}\label{eqn:cg1} \sigma_r(M,\phi) = \text{sign}(W) - \epsilon_r(\widetilde{W}) - \frac{ 2 [F]^2 r(l-r)}{l^2}. 
\end{equation} 
Here  $ \text{sign}(W)$ is the signature of $W$; $\epsilon_r(\widetilde{W})$ is the signature of the intersection form of the $m$--fold cyclic branched cover of $W$ associated to $\overline{\phi}$ restricted to the $\omega_l = e^{2\pi i r/l}$--eigenspace of the action of the generator of the group of deck transformations acting on $H_2(\widetilde{W}, \C)$; and $[F]^2$ is the self-intersection number of $F$.   (That this is a well-defined invariant of the pair $(M, \phi)$ is one of the accomplishments of~\cite{MR520521}. There it is only required that  there is a four-manifold and homomorphism pair $(W,\overline{\phi})$ such that  $\partial(  W, \overline{\phi})  = n(M, \phi)$ for some $n >0$.  In all our work, such a pair exists for $n=1$, so we are restricting to that setting.)

The result~\cite[Lemma 3.1]{MR520521} can be  applied to the case of $S^3_m(K)$ with $\phi$ the quotient  map to $\Z_l$ for a divisor $l$ of $m$, in which case it states:
\begin{equation}\label{eqn:cgthm} \sigma_r(S^3_m(K),\phi) = \text{sign}(m) - \text{sign}((1 - \omega_l^{-r})V + (1 - \omega_l^r)V^{t}) -   \frac{2mr(l-r)}{l^2}.
\end{equation}
Here  sign$(A)$ denotes the signature of a complex hermitian matrix $A$;  if $A$ is one-dimensional, that is if $A = m$ for some real number, sign$(m)$ is simply the sign of $m$.  The matrix  $V$ is a Seifert matrix for $K$.   In standard notation, the   signature of the hermitianized Seifert form is called the {\it Tristram-Levine} $r/l$--signature of $K$, denoted $\sigma_{r/l}(K)$.

\begin{theorem}\label{thm:signthm} Suppose that  $m \ne 0$ and $S^3_m(K)$ bounds a four-manifold $W$ built from $S^1 \times B^3$ by adding a two-handle along a curve representing $l \in H_1(W)$.  Then  $m =\pm  l^2$ for some $l >  0$  and for all  $r$,  $0 <r < l$, 
\[\sigma_{r/l}(K)  = s  - s {2r(l-r)} \pm 1, \] 
where $s = \frac{m}{l^2} = \pm 1$.
\end{theorem}

\begin{proof}   
First observe that the  handle decomposition of $W$ yields a surgery description of $S^3_m(K)$ as $S^3_{0, a}(J_1, J_2)$ for some link $(J_1,J_2)$ and some integer $a$.  The surgery matrix is
\[   \begin{pmatrix} 
0 & l \\
l & a \\
\end{pmatrix}.\]  
In our situation, $a$ will be seen to be $\pm 1$, but for now we simply observe that since the homology of $S^3_m(K)$ is cyclic, $\gcd(a, l) = 1$,  $l^2 = \pm  m$,  and the map induced by inclusion $H_1(S^3_m(K)) \to H_1(W)$ corresponds to the quotient map $\Z_{|m|} \to \Z_l$.

The manifold $W$ can be used to compute $\sigma_r(S_m(K), \phi)$, where $\phi$ is the quotient map $\phi \co \Z_{|m|} \to \Z_l$.  There is no branching surface.   Observe that $W$ is a rational homology ball and so has signature 0.  Also,  $\widetilde{W}$ is built from $S^1 \times B^3$ by adding $l$ two-handles; it follows that each eigenspace is $1$--dimensional, and thus $\epsilon(\widetilde{W}) = \pm 1$.   The definition of the $\sigma_r$ now yields
\[\sigma_r(S^3_m(K), \phi) = \pm 1.\]

Equation~\eqref{eqn:cgthm} then can be written as 
\[ \pm 1 = \text{sign}(m) - \sigma_{r/l}(K) -  s{2 r(l-r)}, \] 
which can be rewritten as
\[ \sigma_{r/l}(K)   = \text{sign}(m) -  {s 2 r(l-r)} \pm 1, \] 
as desired.
\end{proof}

The following corollary is similar to results proved in~\cite{MR1282760} and a related result in~\cite{MR1871306}, which was presented in the case of torus knots.

\begin{corollary} \label{cor:sigobs} If $K$ can be unknotted with a single twist of linking number $l >0$, then for all $r$, $0 <r <  l$, and for $s$ either $1$ or $-1$
\[\sigma_{r/l}(K)  = s  - s {2r(l-r)} \pm 1, \] 
where $s = 1$ or $s=-1$, depending on whether the twist is left-handed or right.
\end{corollary}
\begin{proof}Except for the sign of $s$, this is an immediate consequence.  Suppose that $K$ can be unknotted with a negative twist.  In this case, the three-manifold of interest is $S^3_{l^2}(K)$ and in Equation~\eqref{eqn:cgthm}, the term $\frac{m}{l^2} = 1$.  Similarly for the right-handed twist.\end{proof}
The next result is similar, only we consider the case of $S^3_0(K)$

\begin{theorem} Suppose  that $S^3_0(K)$ bounds a four-manifold $W$ built from $S^1 \times B^3$ by adding a two-handle along a curve representing $0 \in H_1(W)$. Then $W$ is is a definite manifold.   For     all $l >0$,  and   all $r$,  $0 <r \le  l$, 
\[\sigma_{r/l}(K)  = s   \pm 1, \] 
where $s = 1$ if $W$ is positive definite and  $s=-1$ if $W$ is negative definite.  \end{theorem}

\begin{proof}   
In this case, the  handle decomposition of $W$ yields a surgery description of $S^3_0(K)$ as $S^3_{0, a}(J_1, J_2)$ for some link $(J_1,J_2)$ and some integer $a$.  The surgery matrix is
\[
\begin{pmatrix}  
0 &0 \\
0& a \\
\end{pmatrix}.
\] 
The homology of $S^3_0(K)$ is cyclic, so  $a = \pm 1$; we now set $s = a$.

Let $\phi\co H_1(S^3_0(K)) \to \Z_l$ be a surjection.   Then the manifold $W$ can be used to compute $\sigma_r(S_0(K), \phi)$.  There is no branching surface.   In this case, $W$ has signature $\text{sign}(s) = \pm 1$, depending on whether it is positive or negative definite.  Also,  $\widetilde{W}$ is built from $S^1 \times B^3$ by adding $l$ two-handles; it follows that each eigenspace is $1$--dimensional  and   $\epsilon(\widetilde{W}) = \pm 1$.   Thus, the definition of the $\sigma_r$ yields
\[\sigma_r(S^3_0(K), \phi) =  s   \pm 1.\]

Equation~\eqref{eqn:cgthm} then can be written as 
\[  s  \pm 1 = \text{sign}(m) - \sigma_{r/l}(K) -     \frac{(2)( 0) r (l-r)}{l^2}.
\] Here $m = 0$, so this can be  rewritten as
\[ \sigma_{r/l}(K)   = - s \pm 1,
\] 
as desired.
\end{proof}
\begin{corollary} If $K$ can be unknotted with a single twist of linking number $l = 0$, then for all $q > 0$ and all $r$,
\[ \sigma_{r/q}(K)   = - s \pm 1\] 
where $s= 1$ if it is a left-handed twist and $s = -1$ if it is a right-handed twist.

\end{corollary}

\begin{corollary}The positive torus knot $K =T(p,q)$ cannot be unknotted with a positive twist of linking number greater than 1.\end{corollary}
\begin{proof} The signature function satisfies $\sigma_{r/l}(K) \le -2$ for all $r/l > 1/pq$.  (A proof is left to the appendix,  Theorem~\ref{thm:signsigma}.)  On the other hand, if $K$ could be unknotted with a positive twist of some linking number $l \ge 2$, then the terms on the right in Corollary \ref{cor:sigobs} would include nonnegative values.  
\end{proof}

\begin{corollary}
\label{cor:gcd}
Suppose a knot $K$ can be unknotted with twists of linking numbers $l_1, l_2\geq 2$ and signs $s_1$ and $s_2$, respectively.  Then one of the following holds:
\begin{enumerate}

\item $\gcd(l_1, l_2) = 1$,

\item  $l_1 = l_2 = 2$ with $s_1 \neq s_2$, or

\item $l_1 = l_2$ and $s_1 = s_2$.

\end{enumerate}
\end{corollary}

\begin{proof}
Let $l_1, l_2\geq 2$ and gcd$(l_1, l_2) = n \neq 1$.  Then $\frac{1}{n} = \frac{l_1/n}{l_1} = \frac{l_2/n}{l_2}$.  Suppose that $K$ can be unknotted with twists of sign $s_i$ and linking number $l_i$ for $i\in\{1,2\}$.  On the one hand,
\[\sigma_{1/n}(K) = s_1-2s_1\left(\frac{l_1}{n}\right)\left(l_1 - \frac{l_1}{n}\right) \pm 1 = \left(1-2\,{l_1}^2\left(\frac{n-1}{n^2}\right) \right)s_1\pm 1\]
and on the other hand, 
\[\sigma_{1/n}(K) = s_2- 2s_2\left(\frac{l_2}{n}\right)\left(l_2 - \frac{l_2}{n}\right) \pm 1 = \left(1-2\,{l_2}^2\left(\frac{n-1}{n^2}\right) \right)s_2\pm 1.\]
Note that if $s_1 \neq s_2$, these two equations imply that $\sigma_{1/n}(K)$ is both nonnegative and nonpositive.  In this case, $\sigma_{1/n}(K) = 0$ and thus $l_1 = l_2 =2$.   Now assume that $s_1 = s_2$.  From the formulas for $\sigma_{1/n}(K)$ given above, we have
\[ s_1\left(1-2{l_1}^2\frac{n-1}{n^2} \right) - s_1\left(1-2{l_2}^2\frac{n-1}{n^2} \right) = 0 \text{ or } \pm 2.\]
Simplifying,
\[ ({l_2}^2-{l_1}^2) \frac{n-1}{n^2} = 0 \text{ or } \pm 1 \]
and multiplying by $\frac{n^2}{n-1}$
\[{l_2}^2-{l_1}^2 = 0 \text{ or } \pm \frac{n^2}{n-1}.\]
Note that $\frac{n^2}{n-1}$ is an integer only when $n = 2$ and $\frac{n^2}{n-1} =4$.  It is easily checked  that $4$ is not a difference of two squares.  Therefore, we have that that ${l_2}^2-{l_1}^2 = 0$ and  $l_1 = l_2$.
\end{proof}

\begin{example}  The unknot $U$ has $\calu(U) = \{ 2^-,1^-,0, 1, 2\}$.  There are no known example of knots $K$ for which $\{ k^-, (k+1)^+\} \subset \calu(K)$ and $k > 1$.   If such an example exists, then Corollary~\ref{cor:sigobs} implies that the signature function alternates between positive and negative entries at $k$--roots of unity and $(k+1)$--roots of unity.  This implies that the Alexander polynomial has multiple zeroes between these unit roots, and thus we get a bound on the degree of the Alexander polynomial.  This in turn provides a lower bound on the genus of the knot.  Since the result in Corollary~\ref{cor:sigobs}  is stated in terms of a quadratic function, the calculations are not difficult, and results such as the following appear:  If $\{ 3^-, 4^+\} \in \calu(K)$ then $g(K) \ge 9$, and if $\{ 4^-, 5^+\} \in \calu(K)$ then $g(K) \ge 23$

In general, the bound on $g(K)$ is determined by summing quadratic polynomials, and is thus given by a cubic equation.   Having observed this, that cubic can be found explicitly by interpolating the first four values.  We get the following result.
\begin{corollary}
If $\{ k^-, (k+1)^+\} \subset \calu(K)$ then 
\[ 
g(K) \ge \frac{2k^3 + 3k^2 - 11k +6}{6}
\]
\end{corollary} 

We note that if one considers pairs such as $\{k^-, (k+2)^+\}$ the computation becomes unmanageable; in this case the $k$--roots of unity and $(k+2)$-roots of unity do not alternate around the unit circle.

\end{example}

\section{Single twist unknotting: Arf invariant}\label{sec:arf}

If $M$ is a closed three-manifold and $H_1(M, \Z_2) = 0$, then the Rochlin invariant $\mu(M) \in \Z_{16}$ is  defined as follows.  There exists a parallelizable four-manifold $W$ with $\partial W = M$ and $H_1(W, \Z_2) = 0$; $\mu(M) $ is defined to be the signature of the intersection form of $-W$, reduced modulo $16$.

For knots $K\subset S^3$ there is an Arf invariant, $c(K) \in \Z_2$, which can be defined as follows.   If $K$ has determinant  $\det(K)$, then $c(K) = 0$ when $\det(K) \equiv \pm 1 \mod 8$ and $c(K) = 1$ when  $\det(K) \equiv \pm 3 \mod 8$.  This is often stated in terms of the Alexander polynomial, using the fact that $\det(K) = \big| \Delta_K(-1)\big|$.

Background for these invariants is included in~\cite{MR0356022,MR0182965} and especially~\cite[Theorem 2]{MR0402762}, which, in the current setting, implies 
\[\mu(S^3_{n}(K)) \equiv  \mu(L(n, 1)) + 8c(K) \in \Z_{16},\] 
for $n$ odd.
(Note that  in \cite{MR0402762} the invariants take value in $\Q/\Z$, in which $\Z/16$ embeds.)

\begin{theorem} \label{thm:arf obs} If $K$ can be unknotted with a single twist of linking number $l$ with $l$ odd, then:
\[l     \equiv \begin{cases}
\pm 1 \mod 8  &\text{if  } c(K) = 0 \in \Z_2,\\
\pm 3 \mod 8  &\text{if  } c(K) = 1 \in \Z_2. 
\end{cases}   
\]
\end{theorem}

\begin{proof} As described, for instance, in~\cite{MR0515288}, a framed link surgery diagram for $L(n,q)$ is determined by a continued fraction expansion of $n/q$.  There is a diffeomorphism $L(n,1) \cong -L(n, n-1)$, and $n/(n-1)$ has an even continued fraction expansion (in fact, all terms are 2) with $n-1$ terms.  It is an easy exercise to show that the corresponding four-manifold is positive definite, of rank $n-1$.  The result follows quickly by letting $n = l^2$.
\end{proof}

\begin{corollary} Let $K$ be the torus knot $T(p,q)$.  If $p$ and $q$ are odd and $K$ can be unknotted with  single twist of odd linking number  $l$, then $l   \equiv \pm 1 \mod 8$.  If the torus knot $T(2p,q)$ can be unknotted with a single twist of odd  linking number $l$, then $l \equiv \pm q \mod 8$.
\end{corollary}

\begin{proof}  The Alexander polynomial of the torus knot is given by 
\[
\Delta_{T(p,q)}(t)  = \frac{ (t^{pq} -1)(t-1) }{ (t^p-1)(t^q-1)}.
\]
In the case that $p$ and $q$ are both odd, the evaluation at $t= -1$ is immediately seen to be 1.  If $p$ is even, then evaluating $(t^{pq} -1)/(t^p -1)$ at $t=-1$ can be accomplished, for instance, by L'Hospital's rule, and is seen to equal $q$. 
\end{proof}

\begin{example} The torus knot $T(2k, 2k\pm 1)$ can be unknotted with a single twist of linking number $2k \pm 1$.
\end{example}


\section{Summary of Heegaard Floer Theory}\label{sec:hf}

Heegaard Floer theory associates to each knot $K \subset S^3$ a chain complex $\cfk^\infty(K)$ and to each three-manifold $Y$, a collection of chain complexes $\cf^\infty(Y, \spinc)$ (see \cite{MR2113019}). Here $\spinc \in \text{Spin}^c(Y)$, the set of  \Spinc--structures on $Y$.  We will leave the definition of Spin$^c(Y)$ to the references; the key fact that we will be using is that in general there is a correspondence between Spin$^c(Y)$ and $H^2(Y) \cong H_1(Y)$ and in the case of $Y = S^3_m(K)$, there is a natural choice for that correspondence.  In particular, invariants associated to a given \Spinc--structure, such as $d(S^3_m(K), \spinc)$, can be written as  $d(S^3_m(K), i )$, where $i \in \Z $ satisfies $(-|m| +1) /2 \le i \le |m|/2$ and thus uniquely represents an element in $\Z_{|m|}$.  In this section, we will summarize some of the invariants and their properties.

\subsection{Heegaard Floer Knot Invariants $V_k(K)$}  These are integer-valued invariants defined for $k \ge 0$. They satisfy the following properties.

\begin{itemize}

\item $V_k(K) \ge V_{k+1}(K) \ge  V_k(K) -1 $ for all $k\ge 0$.

\item  $V_k(K) = 0$ for all $k \ge g(K)$.
\end{itemize}

In general, these are difficult to compute.  There are two cases in which they are accessible.

\begin{example}{\bf Alternating Knots.}  The Heegaard Floer complex for an alternating knot is determined entirely by the knot's signature, as follows.  If $K$ is alternating and $\sigma(K) \ge 0$, then $V_k(K) = 0$ for all $k \ge 0$.  If $\sigma(K) <0$, then $V_k(K) =\max \{  \lfloor \frac{- \sigma(K)+2(1-k)}{4} \rfloor , 0\}$ for $k \ge 0$.
\end{example}

\begin{example}{\bf Torus Knots.} The $V_k(T_{p,q})$ are determined by the Alexander polynomial. See, for example, \cite{MR3347955}. \end{example}

\subsection{Heegaard Floer Knot Invariants $\nu^+(K)$}
This invariant has a simple definition in terms of the $V_k(K)$:   
\[ {\nu^+}(K) = \min\{n \ | V_n(K) =0\}.\]   
We have $g(K) \ge \nu^+(K)$ for all knots $K$.

\subsection{\bf The Upsilon invariant $\U_K(t)$:}   The Upsilon function $\U_K(t)$ is a piecewise linear function defined for $0 \le t \le 2$.  Some of its key properties are the following.

\begin{itemize}

\item  For all $t \in [0,2]$ and for all knots $K$ and $J$, $\U_{K\cs J}(t) = \U_{K }(t) +\U_{  J}(t)$.

\item $\U_{-K }(t) = - \U_{K}(t)$.

\item  For all nonsingular points $t$, the derivative satisfies $|\U_K'(t)| \le g(K)$.
\end{itemize}

In general, for a particular knot $K$, the invariants $V_k(K)$ offer stronger constraints than does $\U_K$.  However, the additivity of $\U_K$ makes it computable in cases in which computing  the $V_k$ might be difficult.  The proof of the following theorem is left to Appendix~\ref{app:upsilon}, since it calls on some details of Heegaard Floer theory that are not presented in the body of this paper.

\begin{proposition} \label{prop:upsilonVbound}
Let $K$ be a knot and $g=g(K)$ be the genus of $K$.  Then for $t\in[0,2]$ and $s\geq 0$, 
\[-st-2V_s(K)\leq \U_t(K)\leq 
\begin{cases}
-gt-2V_s-2s+2g+2 & t\leq 1-\frac{s}{g} \\
gt -2V_s+2 & t\geq 1-\frac{s}{g} 
\end{cases} \]
\end{proposition}

\subsection{The Heegaard Floer Correction Term, $d(Y, \spinc)$}  Heegaard Floer theory associates to each three-manifold $Y$ with \Spinc--structure $\spinc$, a rational invariant denoted $d(Y, \spinc)$.  
We will need these invariants in the case that $Y$ is surgery on a knot:
\begin{theorem}\label{thm:din2} For $n >0$ and  $0 \le i \le n/2$,  
\[d(S^3_n(K), i) = \frac{(2i - n)^2 - n}{4n} - 2 V_i(K).\]  
\end{theorem}

The main theorems we will use concerning the $d$--invariants are as follows.

\begin{theorem}  Suppose that $|H_1(Y) | = m$ and a given \Spinc--structure $\spinc$ on $Y$ extends to a rational homology ball $W$.  Then $d(Y, \spinc) = 0$.
\end{theorem}

There are some subtleties about determining which \Spinc--structures extend, but if we appropriately choose   identifications of \Spinc\ with $H_1(Y)$ or $H^2(W)$, the key results are easily summarized.  In the background we have that a \Spinc--structure on $Y$ extends to $W$ if and only if the corresponding element in $H^2(Y)$ is in the image of the restriction map from $H^2(W)$.

\begin{theorem} Suppose that $Y = \partial W $, where $H_*(W; \Q) \cong H_*(B^4)$.  Then $|H_1(Y)| = l^2$ for some integer $l >0$ and there exists a coset $H$ of an index $l$ subgroup of $H_1(Y)$ such that  $d(Y, i) = 0$ for all $ i \in H$.
\end{theorem}

\begin{corollary}\label{thm:dinvar}  If $S^3_n(K)$ bounds a rational homology ball, then $n = \pm l^2$ for some $l$.  If $l$ is odd, then  $d(S^3_n(K), kl) = 0$ for $0\le k \le l-1$.  If $l$ is even, then
$d(S^3_n(K), (k+\frac{1}{2})l) = 0$ for $0\le k \le l-1$. 
\end{corollary}
\begin{proof}  A duality argument shows that if a rational homology three-sphere $M$ bounds a rational homology four-ball $X$, then $|H_1(M)| = |\ker(H_1(M) \to H_1(X))|^2$; see, for instance,~\cite{MR900252}.  Thus, we write $n = l^2$.   It follows from the original results of Ozsv\'ath-Szab\'o~\cite{MR1957829}  that for $l$ values of $i$, $d(S^3_n(K), i) = 0$.   It is now an exercise in arithmetic, using Theorem~\ref{thm:din2},  to show that the only integer values occur at $i = kl$ for $l $ odd and at $i = (k +\frac{1}{2})l$ for $l$ even.
\end{proof}

\section{Single twist unknotting: Heegaard Floer obstructions}\label{sec:hf2}

In~\cite{MR3347955}, the Heegaard Floer $d$--invariants of three-manifolds manifolds of the form $S^3_{k^2}(K)$ were studied in the case of algebraic knots.  
Aceto and Golla~\cite{MR3604383} expanded on this,  undertaking an extensive study of the question of, for a given knot $K$, which of the manifolds $S^3_{p/q}(K)$ bound rational balls.  Many of the results of this section are built from special cases of what appears there.  For instance, their theorem that if $S^3_{l^2}(K)$ and $S^3_{m^2}(K)$ both bound rational homology balls, then $l$ and $m$ are consecutive. Our Theorem~\ref{thm:v+ bound} follows immediately, showing that if $0<l<m$ and $\{l^-, m^-\} \subset \calu(K)$, then $l - m =1$.  We will include proofs of the results we need for two reasons: in our setting the arguments are fairly straightforward and accessible, and the arguments provide access to stronger results in the case of the unknotting problem.

We begin with the following, which follows readily from~\cite{MR3347955} and is stated explicitly in the context of unknotting twists by Sato~\cite{MR3823992}.

\begin{theorem}\label{thm:vals} Suppose that $K$ can be unknotted with a negative twist of linking number $l >0$.  
\begin{itemize}

\item If $l$ is odd, then there is an $\alpha \ge 0$ such that $l = 2\alpha +1$ and for all $0\le k \le \alpha$,
\[   V_{kl}(K) = (\alpha - k)(\alpha - k +1)/2.\]

\item If $l$ is even, then there is  a $\beta \ge 0$ such that $l = 2\beta +2$ and for all $0\le k \le \beta$,
\[ V_{(k+\frac{1}{2})l}(K) = (\beta - k)(\beta - k +1)/2.\]
\end{itemize}
\end{theorem}

\begin{proof} 
Suppose that $K$ can be unknotted with a negative twist of linking number $l>0$.  Then $S_{l^2}^3(K)$ bounds a rational homology ball and we can apply Corollary~\ref{thm:dinvar}.  If $l$ is odd, then $l=2\alpha +1$ for some $\alpha \geq 0$.  Theorems \ref{thm:din2} and \ref{thm:dinvar} imply that for $0\leq k \leq \alpha$, 
\begin{align*}
V_{kl}(K) &= \frac{(2kl-l^2)^2-l^2}{8l^2}\\
&= \frac{(2k-2\alpha -1)^2-1}{8} \\
&= \frac{(\alpha - k)(\alpha - k+1)}{2},
\end{align*}
as desired.

Similarly, if $l$ is even, then $l = 2\beta +2$ for some $\beta\geq 0$.  For $0\leq k\leq \beta$, we have
\begin{align*}
V_{\left(k+\frac{1}{2}\right)l}(K) &= \frac{(2\left(k+\frac{1}{2}\right)l-l^2)^2-l^2}{8l^2}\\
&= \frac{(2\left(k+\frac{1}{2}\right)-2\beta -2)^2-1}{8} \\
&= \frac{(2k-2\beta -1)^2-1}{8} \\
&= \frac{(\beta - k)(\beta - k+1)}{2},
\end{align*}
as desired.
\end{proof}

This theorem places unexpectedly strong constraints on the possible values of $l$.  Recall ${\nu^+} = {\nu^+}(K) = \min\{n \ | V_n(K) =0\}$.

\begin{theorem} \label{thm:v+ bound} For a knot $K$, there  are at most two positive values of $l$ for which $K$ can be unknotted by a negative twist of linking number $l$. If $K$ can be unknotted using negative twists of two different linking numbers, then ${\nu^+}(K) = \gamma(\gamma+1) /2$ for some $\gamma$ and the two values of $l$ are $l_1 = (1 + \sqrt{1 + 8{\nu^+}})/2)$ and $l_2 = l_1 +1$.  If $\nu^+$ is not of this form, there is at most one possible value for  $l$, and it is given by the ceiling, $\lceil  (1 + \sqrt{1 + 8{\nu^+}})/2)  \rceil $.

\end{theorem}

\begin{proof} 
For odd $l$, if we let $k = \alpha -1$  we see that $V_{(\alpha -1)(2\alpha +1)}=1$.  Letting $k= \alpha$, we have    $V_{\alpha(2\alpha+1)}=0$.
Thus,  for $l$ odd, 
\[(\alpha-1)(2\alpha+1) < {\nu^+}\le \alpha(2\alpha +1).\]

For even $l$, if we let $k = \beta -1$ we see that $V_{(\beta -\frac{1}{2})(2\beta +2)} = 1$.  If we let $k = \beta$ we see that $V_{(\beta +\frac{1}{2})(2\beta+2)} =0$.     Thus,  we arrive at the inequalities \[\left(\beta -\frac{1}{2}\right)(2\beta +2) < {\nu^+}\le \left(\beta +\frac{1}{2}\right)(2\beta+2).\]

In either case, these are quadratic in $\alpha $ or $\beta$ and the bounds on each are determined using the quadratic formula.  Expressing either in terms of $l$ (and recalling that $\alpha, \beta\geq 0$) yields the same inequality:

\[ \frac{ 1 + \sqrt{1 + 8{\nu^+}}}{2} \le l <  \frac{ 3 + \sqrt{9 + 8{\nu^+}}}{2}  \]

For $\nu^+ >0$, the difference of these bounds is strictly between 1 and 2.  If $\nu^+ = 0$, then the difference of these bounds is exactly 2.  In either case, the interval can contain at most two integers. The left endpoint is an integer exactly when ${\nu^+} = \gamma(\gamma+1)/2$ for some integer $\gamma$.  In this case the interval contains two integers.  If $\nu^+$ is reduced by 1, then the right endpoint becomes an integer.  In this case, since the right endpoint is not included in the interval, there is only one integer in the interval.
\end{proof}

\begin{example} Let $K = T(7,8)$.  Then $K$ can be unknotted with a negative twist of linking number 7, and we have $V_0(K) = 6, V_7(K) = 3, V_{14}(K) = 1$, and $V_{21}(K) = 0$.

We also have that $K$ can be unknotted with a negative twist of linking number $8$, and $V_4(K)= 6$, $V_{12}(K) = 3$, $V_{20}(K) = 1$, and $V_{28}(K) = 0$.
\end{example}

\begin{corollary}\label{corr1}
For any knot $K$, there are at  most three values of $l$ such that  $K$ can be unknotted with a single  positive twist of linking number $l$.  Similarly, there are at most three values of $l$ such that  $K$ can be unknotted with a single negative  twist of linking number $l$.

\end{corollary}
\begin{proof} There are two possible positive linking numbers for negative twists.  Considering mirror images, we see there are at most two possible positive linking numbers for negative twists.  Finally, there is the possibility of unknotting with a linking number 0 twist.
\end{proof}

Thus, for a given knot $K$, $|\calu(K)|\leq 6$.  This combined with Corollary \ref{cor:gcd} implies the following result.

\begin{corollary}\label{cor:6}
If $|\calu(K)| = 6$, then $\calu(K) = \{2^-, 1^-, 0^-, 0^+, 1^+, 2^+\}$.
\end{corollary}

The unknot realizes this unknotting set.  Whether or not there is a nontrivial knot with this unknotting set is unknown.  Such a knot would have $\nu^+(K) = \nu^+(-K) = 0$.


\section{Further Heegaard Floer obstructions}\label{sec:hf3}

Suppose that $K$ can be unknotted with a single negative twist of linking number $l$.  In terms of surgery diagrams, this means that blowing up a $-1$ unknotted circle that has linking number $l$ with $K$ creates a link $(U, K^*)$, where $K^*$ is unknotted.  Consider $Y = S^3_{l^2+1}(K)$.  Since blowing up a $-1$ lowers framings by the square of the linking number, we see that $Y$ has a surgery description given by $-1$--surgery on $U$ and $1$--surgery on $K^*$.   We can modify the surgery description again by blowing down the $+1$.  This has the effect of lowering the framing on the other component by $l^2$.  Thus, we see that $Y$ can be described by $(-l^2 -1)$--surgery on a second knot, which we denote by $J$.  The following lemma is easily proved by considering the blowup and blowdown.

\begin{lemma} If $K$ can be unknotted with a single negative twist of linking number $l$, then there is a knot $J$ and an orientation preserving  homeomorphism from $S^3_{l^2 +1}(K)$ to $S^3_{-l^2  -1}(J )$.  On homology, this homeomorphism   carries the first homology class represented by the meridian of $K$ to $l$ times the first homology class represented by the meridian of $J$.  
\end{lemma}

To compute $d$--invariants, we will want to reduce integers modulo $l^2 +1$ appropriately.

\begin{definition}   For $a, n \in \Z$ with $n>1$, we define $a_n$ to be the least nonnegative number for which   $a - a_n$ is divisible by $n$.  We define 
\[[a]_n = \left|   \left(a+ \frac{n-1}{2} \right)_n - \frac{n-1}{2} \right|.\]
\end{definition}

\begin{example} $\  $
\begin{itemize}

\item $[0]_4 = 0$ \hskip.2in $[1]_4 = 1$ \hskip.2in $[2]_4 = 2$ \hskip.2in $[3]_4 = 1$.
\vskip.1in

\item $[0]_5 = 0$ \hskip.2in $[1]_5 = 1$ \hskip.2in $[2]_5 = 2$ \hskip.2in $[3]_5 = 2$ \hskip.2in $[4]_5 = 1$.
\end{itemize}
\end{example}

\begin{theorem}If $K$ can be unknotted with single negative twist of linking number $l$, then there exists a knot $J$ such that
\[ d( S^3_{l^2 +1}(K), i) = d(S^3_{-l^2 -1} (J) , [li+ \beta]_{l^2+1} )\]
for all integers $i$ satisfying $0 \le i < \frac{l^2 + 1}{2}.$  Here $\beta = 0$ if $l$ is even and $\beta = \frac{l^2+1}{2} $ if $l$ is odd.
\end{theorem}

\begin{proof} The only issue that requires proof is the term $\beta$ that appears.  The issue arises because of how the \Spc--structures are parameterized with integers. In the case that $n = l^2 +1$ is odd there is a unique Spin--structure on $S^3_{-l^2 -1} (K)$ and this determines which \Spc--structure is denoted $\spinc_0$.  However, if $n$ is even, there are two Spin structures, one of which corresponds to $\spinc_0$ and the other to $\spinc_k$, where $k = \frac{l^2 +1}{2}$.
There is a simple means to rule out one of the possibilities:   if $n = l^2 +1$ is even, and $\beta =0$, then 
\[ d( S^3_{l^2 +1}(K), i) \ne d(S^3_{-l^2 -1} (J) , [li+ \beta]_{l^2+1} ) \mod \Z.\]

\end{proof}

It is simpler to have both surgery coefficients positive, so we consider the mirror image of $J$ and use the symmetry to $d$--invariant under conjugation to conclude the following.

\begin{theorem}If $K$ can be unknotted with single negative twist of linking number $l$, then there exists a knot $J'$ such that
\[ d( S^3_{l^2 +1}(K), i) = -d(S^3_{ l^2+1} (J') , [li+\beta]_{l^2 +1})\]
for all integers $i$ satisfying $0 \le i < \frac{l^2 + 1}{2}.$ Here $\beta = 0$ if $l$ is even and $\beta = \frac{l^2+1}{2} $ if $l$ is odd.
\end{theorem}

We can now apply Theorem~\ref{thm:din2} to get the following result.

\begin{theorem} Suppose that $K$ can be unknotted with single negative twist of linking number $l$ and  $n = l^2+1$.   Then there exists a knot $J'$  such that for all $i$ satisfying  $0 \le i \le n/2$,

\[ \frac{(2i - n)^2 - n}{4n} - 2 V_i(K) =  - \frac{(2[li+\beta]_{n} - n)^2 - n}{4n} + 2 V_{[li+\beta]_{n}}(J').\]  
Here $\beta = 0$ if $l$ is even and $\beta = n/2 $ if $l$ is odd.

\end{theorem}

Rearranging the terms of this expression, we have:

\begin{corollary}\label{cor:V} Suppose that $K$ can be unknotted with single negative twist of linking number $l$ and  $n = l^2+1$.   Then there exists a knot $J'$  such that for all $i$ satisfying  $0 \le i \le n/2$,

\[   V_{[li+\beta]_{n}}(J') =  -\frac{1}{4}   -\frac{i}{2}-  \frac{[li+\beta]_{n}}{2} + \frac{i^2}{2n} +  \frac{[li+\beta]_{n}^2}{2n}  + \frac{n}{4} -   V_i(K).   \]

Here $\beta = 0$ if $l$ is even and $\beta = n/2 $ if $l$ is odd.
\end{corollary}

To apply this corollary, recall the following property of the $V_i$ invariants:
\begin{equation} \label{eq:vi prop} 0\leq V_i(K)-V_{i+1}(K) \leq 1. \end{equation}
Let $n = l^2+1$ and define 
\[j(i) = [li+\beta]_{n}\text{  and  } s(i) = -\frac{1}{4}   -\frac{i}{2}-  \frac{j(i)}{2} + \frac{i^2}{2n} +  \frac{j(i)^2}{2n}  + \frac{n}{4}\] 
so that, for each $i \in \{0, \dots, n/2\}$, we have 
\begin{equation} \label{eq:vi1} V_{j(i)}(J') = s(i)-V_i(K). \end{equation}
If $j(i) < n/2$, we also have 
\begin{equation} \label{eq:vi2} V_{j(i)+1}(J') = s(i')-V_{i'}(K) \end{equation}
for some $i'$.
Substituting Equations \eqref{eq:vi1} and \eqref{eq:vi2} into Equation \eqref{eq:vi prop} and rearrranging, we get:
\begin{itemize}

\item if $i'<i$, then $\displaystyle{s(i')-s(i)\leq V_{i'}(K)-V_i(K)\leq s(i')-s(i)+1}$, and

\item if $i<i'$, then $\displaystyle{s(i)-s(i')-1\leq V_{i}(K)-V_{i'}(K)\leq s(i)-s(i')}.$
\end{itemize}
This process yields $n/2$ inequalities.  Due to Equation \eqref{eq:vi prop}, some of these inequalities will be redundant.

\begin{example}  Consider the case of $l=0$ and $i=0$.  Then Corollary~\ref{cor:V} implies that $V_0(J') = - V_0(K)$.  Since these are positive, we have the following theorem, first proved by Sato~\cite{MR3823992}.

\begin{theorem}  If $K$ can be unknotted with single twist of linking number $l = 0$, then $\nu^+(K) = 0$.

\end{theorem}
\end{example}

\begin{example} Consider the case of $l = 4$.  We have the following table of values.

\[
\begin{array}{|c|c|c|c|c|c|c|c|c|c|c|c|c|c|c|c|c|c|} 
\hline
	i & 0&1&2&3&4&5&6&7&8 \\
	\hline
	j(i) & 0& 4& 8& 5& 1& 3& 7& 6& 2\\
	\hline
	s(i) & 4&2&1&1&2&1&0&0&1 \\
	\hline
	\end{array} \]
We conclude that if a knot $K$ can be unknotted with a single negative twist of linking number $l=4$, then the following inequalities must be satisfied.  (The redundant inequalities have been removed.)

\begin{align*}
1 \le V_0(K)  &- V_4(K)   \le 2 \\
V_4(K) &- V_8(K)  \le 1 \\
V_5(K) &- V_8(K)  \le 1 \\
1\le V_1(K) &- V_5(K)  \le 2 \\
V_1(K) &- V_3(K)  \le 1 \\
V_3(K) &- V_7(K)  \le 1\\
V_2(K) &- V_6(K)  \le 1.
\end{align*}
\end{example}

\begin{example} \label{ex:3,17}
In the case of $l=7$, a similar computation yields 23 inequalities after 6 redundant ones have been removed.  (This is a tedious computation which we omit.)  We now compare this to the values given by Theorem \ref{thm:vals}:  if a knot $K$ can be unknotted with a negative twist of linking number $l=7$, then 
\[V_0(K) = 6, V_7(K) = 3, V_{14}(K) = 1, V_{21}(K) = 0.\]  
Note that this implies that $V_i(K) = 0$ for all $i\geq 21$.  Imposing these restrictions reduces our initial list of inequalities to a list of 14 inequalities (5 of which consist of a single sub-inequality).  This indicates that, for a fixed linking number $l$, the construction may yield finer information than Theorem \ref{thm:vals}.  For example, the inequality 
\[1\leq V_{9}(K)-V_{16}(K) \leq 2 \] 
remains, while Theorem \ref{thm:vals} tells us that 
\[1\leq V_9(K) \leq 3 \text{ and } 0\leq V_{16}(K) \leq 1,\] 
implying that 
\[0\leq V_9(K)-V_{16}(K) \leq 3,\] 
a broader range.

Consider the knot $K = T(3,17)$.  One can compute the following table of $V_i$ invariants.
\[
\begin{array}{|c|c|c|c|c|c|c|c|c|c|c|c|c|c|c|c|c|c|} 
\hline
	i & 0&1&2&3&4&5&6&7&8&9&10&11&12&13&14&15&16 \\
	\hline
	V_i(K) & 6&5&5&5&4&4&4&3&3&3&2&2&2&1&1&1&0 \\
	\hline
	\end{array} \]
\noindent We see that Theorem \ref{thm:vals} cannot rule out a negative twist of linking number 7, while our work above does:
\[V_9(K)-V_{16}(K) = 3>2.\]
Furthermore, a negative twist of linking number 7 is not ruled out by Corollary \ref{cor:sigobs} (signature obstruction) or Theorem \ref{thm:arf obs} (Arf obstruction). 
\end{example}


\section{Heegaard Floer obstructions related to the Upsilon invariant}\label{sec:hf5}

As we have seen, the invariants $V_i(K)$ provide strong obstructions for a given integer $l $ to satisfy $l \in \calu(K)$.  However, these invariants can be difficult to compute; for instance, they do not behave additively under connected sums of knots.  In this section, we will apply Theorems \ref{thm:vals} and \ref{thm:v+ bound} along with  Proposition \ref{prop:upsilonVbound} to determine bounds on  $l$ for which the specific computation of the $V_k$ would be difficult.

\begin{example}
Consider the knot $K = T(2,25)-T(3,8)$.  This knot has $\tau (K) = 5$ and $g_4(K) =7$  (see \cite{MR3622312}).  Since $\tau(K) \leq \nu^+(K)\leq g_4(K)$ (see \cite{MR3523259}), Theorem \ref{thm:v+ bound} implies
\[\frac{1+\sqrt{1+8\tau(K)}}{2}\leq l < \frac{3+\sqrt{9+8g_4(K)}}{2}\]
and we have $4\leq l \leq 5$.

From Theorem \ref{thm:vals}, we know the following:
\begin{align}
\text{If } l=4, &\text{ then } V_2 = 1, V_6 = 0.  \\
\text{If } l=5, &\text{ then } V_0 = 3, V_5 = 1, V_{10} = 0. 
\end{align}
Proposition~\ref{prop:upsilonVbound} yields a list of restrictions on the Upsilon function of $K$:
\begin{itemize}

\item If $l=4$, then $\displaystyle{\U_K(t)\geq \max \{-2t-2, -6t\} = \begin{cases} -6t & t\leq 1/2 \\ -2t-2 & t\geq 1/2 \end{cases}.}$

\item If $l=5$, then $\displaystyle{\U_K(t)\geq \max \{-6, -5t-2, -10t\}= \begin{cases} -10t & t\leq 2/5 \\ -5t-2 & 2/5 \leq t\leq 4/5 \\ -6 & t\geq 4/5 \end{cases}.}$
\end{itemize}
On the other hand, Upsilon functions of torus knots are easily computed~\cite{MR3667589}.  We have that
\[\U_{T(2,25)}(t) = \begin{cases} -12t & 0\leq t \leq 1 \\ 12t-24 & 1\leq t \leq 2 \end{cases}
\;\;   \text{ and } \;\;   
\U_{T(3,8)}(t) =  \begin{cases} -7t & 0\leq t\leq 2/3 \\ -t-4 & 2/3\leq t \leq 1 \\ t-6 & 1 \leq t \leq 4/3 \\ 7t-14 & 4/3 \leq t \leq 2 \end{cases}  \]
and so
\[\U_{T(2,25)-T(3,8)}(t) = \begin{cases} -5t & 0\leq t\leq 2/3 \\ -11t+4 & 2/3 \leq t \leq 1 \\ 11t-18 & 1\leq t \leq 4/3 \\ 5t-10 & 4/3\leq t \leq 2     \end{cases}.\]
Comparing this to the restrictions above, we have an obstruction when $t=1$ for both $l=4$ and $l=5$.  We conclude that the knot $K = T(2,25)-T(3,8)$ cannot be unknotted with a negative twist of linking number $l>0$.
\end{example}

\section{Obstructions from the Heegaard Floer homology of double branched covers}\label{sec:hf6}

In Section~\ref{sec:link} we explored how the linking form on the two-fold branched cover of a knot $K$ provides obstructions to unknotting with a single twist.  The Heegaard Floer correction term, the $d$--invariant, can be thought of as a $\Q$--valued lifting of the self-linking form, which takes values in $\Q/\Z$.  Thus, as we now describe, when the linking form obstructions vanish, it is possible for the lifted invariants to provide non-trivial obstructions.

The needed result from Heegaard Floer theory is the following.

\begin{theorem}[\cite{MR1957829,MR2388097}] \label{thm:dinvbound}
Let $Y$ be a rational homology three-sphere which is the boundary of a simply-connected positive-definite four-manifold $X$  with $|H^2(Y;\Z)|$ odd.  Let the intersection pairing of $X$ be represented in a basis by the matrix $Q$. Define a function 
\[m_Q: \Z^r/Q(\Z^r) \rightarrow \Q \]
by
\[m_Q(g) = \min \left.\left\{ \frac{\xi^TQ^{-1}\xi - r}{4} \,\right| \xi\in \text{Char}(Q), [\xi]=g\right\} \] 
where $\text{Char}(Q)$ is the set of characteristic covectors for $Q$. Then there exists a group isomorphism 
\[\phi : \Z^r/Q(\Z^r) \rightarrow \text{Spin}^c(Y) \]
with
\[m_Q(g)\geq d(Y, \phi(g)), \]
\[ \text{ and } \;\; m_Q(g) \equiv  d(Y, \phi(g))\;\; (\text{mod } 2) \]
for all $g\in \Z^r/Q(\Z^r).$
\end{theorem}
Note that $\text{Char}(Q)$  corresponds to  the set of first Chern classes of \Spc--structures for $X$ and we are using the identification of \Spc--structures on $Y$ with $H^2(Y)$.  We will also use the fact that 
\[ \text{Char}(Q) = \{\xi = (\xi_1, \xi_2, \dots , \xi_r) \in \Z^r \mid \xi_i \equiv Q_{ii}\}.\]
According to~\cite{MR2388097}, to compute $m_Q$  it suffices to consider characteristic covectors such that 
\[-Q_{ii}\leq \xi_i \leq Q_{ii}-2,\]
which, for rank 2 forms, makes the computation fast, even for relatively large values of the $Q_{ii}$.

To illustrate the use of application of these results to the untwisting problem, we begin with a basic example.

\begin{example}
Suppose a knot $K$ satisfies  $\sigma(K) = -2$ and $\det(K) = 3$; for instance the trefoil knot  or any knot with the same Seifert form.    Suppose that, like the trefoil,  $K$ can be unknotted with a  negative twist of linking number 2. Corollary \ref{cor:two-fold} implies that $M_2(-K)$ bounds a simply-connected, positive-definite four-manifold $N$ with  $b_2(N) = 2$ and intersection pairing 
\[Q= \begin{pmatrix} a & b \\ b & a \end{pmatrix} \]  with determinant $3$.  There are only two such matrices:
\[ \begin{pmatrix} 2 & 1 \\ 1 & 2 \end{pmatrix} \text{ and }  \begin{pmatrix} 2 & -1 \\ -1 & 2 \end{pmatrix}. \]
These differ by a change of basis, so we consider only $Q = \begin{pmatrix} 2 & 1 \\ 1 & 2 \end{pmatrix}$.    Then we have the quotient map $\phi \co \Z^2 \to \Z^2/Q(\Z^2) \cong \Z/3\Z \cong H^2(Y)$; the cosets of the kernel have representatives 
$\displaystyle{g_0 = \begin{pmatrix} 0\\0 \end{pmatrix}, } $  
$\displaystyle{  g_1 = \begin{pmatrix} 0\\1 \end{pmatrix},} $ and  
$ \displaystyle{   g_2 = \begin{pmatrix} 0\\2 \end{pmatrix}}$.   To compute $m_Q$ we consider the subset of characteristic covectors
\[ \left\{ \begin{pmatrix} -2\\-2 \end{pmatrix}, \begin{pmatrix} -2\\0 \end{pmatrix}, \begin{pmatrix} 0\\0 \end{pmatrix}, \begin{pmatrix} 0\\-2 \end{pmatrix} \right\}. \]
A quick computation shows that only $ \begin{pmatrix} 0\\0 \end{pmatrix}$ is in the coset of $ \begin{pmatrix} 0\\0 \end{pmatrix}$, only $ \begin{pmatrix} -2\\-2 \end{pmatrix}$ is in the coset of $ \begin{pmatrix} 0\\2 \end{pmatrix}$, and $ \begin{pmatrix} -2\\0 \end{pmatrix}$ and $ \begin{pmatrix} 0\\-2 \end{pmatrix}$ are both in the coset of $ \begin{pmatrix} 0\\1 \end{pmatrix}$.
We compute that
\[m_Q(g_0) =-\frac{1}{2} \;\;\text{ and }\;\; m_Q(g_1) = m_Q(g_2) = \frac{1}{6}. \] 
Thus the theorem implies that  $\phi$ satisfies
\[ -\frac{1}{2} \geq d(\Sigma(-K), \phi(g_0)), \]
\[ \frac{1}{6} \geq d(\Sigma(-K), \phi(g_1)), \]
\[ \frac{1}{6} \geq d(\Sigma(-K), \phi(g_2)). \]
Note that $\phi(g_0)$ necessarily represents the Spin--structure on the two-fold branched cover.  In the case where $K$ is the trefoil, $\Sigma(-K) = -L(3,1)$ and the three bounds are sharp.

Manolescu and Owens~\cite{MR2363303} computed that for the untwisted right-handed Whitehead double of the trefoil, $J = \text{Wh}^+(T_{2,3}, 0)$, the $d$--invariant of the Spin--structure on its two-fold branched cover is $-4$.   This knot also satisfies $\Delta_J(t) = 1$, and thus $\det(J) = 1$ and $\sigma(J) = 0$.  Thus, the calculation shows that $T_{2,3} \cs J$ cannot be unknotted with a negative twist of linking number 2, and this cannot be obstructed by any classical knot invariant.
\end{example}

\begin{example}
Consider the knot $K =9_5$.  This is a two-bridge knot with $\sigma(K) =-2$,  $\det(K) = 23$, and $\Sigma(K) = L(23,17)$.  We will show that $2^-\not\in \mathcal{U}(K)$. Suppose that $K$ could be unknotted with a single negative twist of linking number $2$.  Then Corollary \ref{cor:two-fold} implies that $\Sigma(-K)$  bounds a simply-connected, positive-definite four-manifold $N$ with  $b_2(N) = 2$ and intersection pairing 
\[Q= \begin{pmatrix} a & b \\ b & a \end{pmatrix} \] with determinant 23. Up to change of basis, there is only one such matrix:
\[ Q = \begin{pmatrix} 12 & 11 \\ 11 & 12 \end{pmatrix}. \]
Note that this matrix is not ruled out by the methods of Section~\ref{sec:link}. We have the quotient map $\psi:\Z^2 \rightarrow \Z^2/Q(\Z^2) \cong \Z/23\Z\cong H^2(\Sigma(-K))$; the cosets of the kernel have representatives $g_i = \begin{pmatrix} 0 \\ i \end{pmatrix}$ for $0\leq i \leq 22$.  We compute that, in particular, \[m_{Q}(g_4) = -\frac{19}{46}. \]
In \cite{MR1957829}, Ozs\'ath and Szab\'o give a formula for computing the $d$-invariants for $-L(p,q)$.  We find that the set of $d$-invariants for $Y = \Sigma(-K) = -L(23, 17)$ are
\begin{equation*}
\begin{split} 
& \left\{ \frac{29}{46}, \frac{1}{46}, -\frac{11}{46}, -\frac{7}{46}, \frac{13}{46}, \frac{49}{46}, \frac{9}{46}, -\frac{15}{46}, -\frac{1}{2}, -\frac{15}{46}, \frac{9}{46}, \frac{49}{46}, \frac{13}{46}, \right. \\
& \left. \;\; -\frac{7}{46}, -\frac{11}{46}, \frac{1}{46}, \frac{29}{46}, \frac{73}{46}, \frac{41}{46}, \frac{25}{46}, \frac{25}{46}, \frac{41}{46}, \frac{73}{46} \right\}. 
\end{split}
\end{equation*}
Among these, the only value which is congruent to $-\frac{19}{46}$ modulo $2\Z$ is $\frac{73}{46}$.  Thus any isomorphism $\phi$ such that 
\[d(Y, \phi(g_4)) \equiv m_{Q}(g_4) \text{ (mod }2)\]
would not satisfy 
\[ d(Y, \phi(g_4)) \leq m_{Q}(g_4).\]
By Theorem \ref{thm:dinvbound}, we have reached a contradiction.  Therefore, $K = 9_5$ cannot be unknotted with a single negative twist of linking number $2$. 
\end{example}


\section{Obstructions for alternating knots}\label{sec:alt}

In \cite{MR3134023}, Petkova showed that the minus version of the Heegaard Floer knot complex, $\cfk ^-(K)$, of a thin knot $K$ is completely determined by its Ozsv\'ath-Szab\'o tau invariant and its Alexander polynomial.  Alternating knots are a subset of thin knots and it is known (see \cite{MR2026543}) that the tau invariant of an alternating knot is determined by its knot signature.  It follows that, for each alternating knot $K$, its $V_i(K)$ invariants are equal to those of some $T(2,n)$ torus knot, determined by $\sigma(K)$:

\begin{lemma} \label{lem:VforT(2,k)}
If $K = T(2,2k+1)$, then 
\[V_i(K) = \begin{cases}
\hfil \frac{k}{2}  - \left\lfloor \frac{i}{2} \right\rfloor & \text{ if $k$ is even and $0 \leq i<k$} \\
\frac{k+1}{2} - \left\lceil \frac{i}{2} \right\rceil & \text{ if $k$ is odd and $0 \leq i<k$} \\
\hfil 0 & \text{ if $i\geq k$ }.
\end{cases}\]
\end{lemma}

\begin{proposition} \label{prop:alt}
Suppose that $K$ is an alternating knot.  If $K$ can be unknotted with a negative twist of linking number $l>0$, then $l\in \{1,2,3, 4\}$.
\end{proposition}

\begin{proof}
The proof follows from the following observation: Lemma \ref{lem:VforT(2,k)} implies that the $V_i$ values decrease linearly, while Theorem \ref{thm:vals} implies that the $V_i$ values decrease quadratically.  More precisely, fix a linking number $l>0$.  From Lemma \ref{lem:VforT(2,k)} we have that, for $i\geq 0$ and $i+l\leq k$, 
\begin{equation} \label{eq:vi diff} V_i - V_{i+l} \in \left\{ \left\lfloor \frac{l}{2} \right\rfloor, \left\lfloor \frac{l}{2} \right\rfloor +1 \right\}. \end{equation}

From Theorem \ref{thm:vals}, we have that if $l$ is odd, $j\geq 0$, and $j+1\leq \frac{l-1}{2}$, then 
	\[V_{jl}-V_{jl+l} = \frac{l-1}{2}-j.\]
If, in addition, $jl+l\leq k$, then we can apply Equation \eqref{eq:vi diff} to see that $j=0$. This implies that Theorem \ref{thm:vals} can determine at most 2 nonzero $V_i(K)$ values.  Thus $l \in \{1,3,5\}$.
We repeat this process when $l$ is even.  From Theorem \ref{thm:vals}, if $l$ is even, $j\geq 0$, and $j+1\leq \frac{l-2}{2}$, then 
	\[V_{\left(j+\frac{1}{2}\right)l}-V_{\left(j+\frac{1}{2}\right)l+l} = \frac{l-2}{2}-j.\]  If, in addition, $\left(j+\frac{1}{2}\right)l+l\leq k$, then  we can apply Equation \eqref{eq:vi diff} and reach a contradiction for all values of $l$.
	
Thus, if Theorem \ref{thm:vals} determines at least two nonzero $V_i(K)$ values, or if it determines one nonzero value in addition to determining that $V_k(K) = 0$, then it must be that $l = 1, 3,$ or 5.  The remaining cases are when Theorem \ref{thm:vals} determines
\begin{enumerate}

\item one nonzero value and one $V_i(K) = 0$ where $i>k$.

\item only one value.

\end{enumerate}	
In the notation of Theorem \ref{thm:vals}, these are the cases where $\alpha, \beta = 1$ or $0$ respectively, or $l = 1, 2, 3, $ or $4$.  Therefore we must have that $l \in \{1,2,3,4,5\}$.

Note that if $l = 5$, then Theorem \ref{thm:vals} implies that $V_0(K) = 3$.  Comparing this to Lemma \ref{lem:VforT(2,k)}, we conclude that $V_i(K) = V_i(T(2, 13))$.  In this situation, a computation similar to that in Example \ref{ex:3,17} yields a contradiction.  Thus an alternating knot cannot be unknotted with a negative twist of linking number 5.
\end{proof}

Combining this result with those of the previous sections, we can narrow down the possible values in $\mathcal{U}(K)$ for $K$ an alternating knot.
\begin{theorem} \label{thm:alt table}
Suppose that $K$ is an alternating knot, then $\mathcal{U}(K)$ is a subset of one of the following, determined by $\sigma(K)$.
\[
\begin{array}{c|c} 
	\sigma(K) &  \\
	\hline
	0 & \{2^-, 1^-, 0^-, 0^+, 1^+, 2^+\} \\
	\hline
	\pm 2 & \{1^{\mp}, 0^\mp, 2^\pm, 3^\pm \} \\
	\hline
	 \pm 4 & \{1^{\mp},  3^\pm \} \\
	 \hline
	 \pm 6, \pm 8 & \{1^{\mp},  4^\pm \}\\
	 \hline
	 >8 & \{1^{-} \} \\
	 \hline
	 <-8& \{1^{+} \}\\
	\end{array} \]
\end{theorem}

\begin{example}
Here we show that $K = 12a_{369}$ cannot be unknotted with a single twist.  We first note that $K$ is alternating with $\sigma(K) = 6$ so that $\mathcal{U}(K) \subseteq \{1^{-},  4^+ \}$.  Because $\text{Arf}(K) = 1$, Theorem \ref{thm:arf obs} implies that $1^- \not\in \mathcal{U}(K)$.  If $4^+ \in \mathcal{U}(K)$, then Corollary \ref{cor:sigobs} implies that $\sigma_{1/4}(K) = 5\pm 1$.  However, $\sigma_{1/4}(K) = 2$ and thus $\mathcal{U}(K) = \{\}$.
\end{example}

\begin{proof}[Proof of Theorem \ref{thm:alt table}]
Suppose that $K$ is an alternating knot. Let $k = - \sigma(K)/2$. Then for each $i$, we have that $V_i(K) = V_i(T(2,2k+1))$ and, because $-K$ is also alternating, $V_i(-K) = V_i(-T(2, 2k+1)) = V_i(T(2, 2(-k-1)+1))$.  Note that if $K$ can be unknotted with a twist of linking number $l$, then $-K$ can be unknotted with an opposite twist of the same linking number.  From Lemma \ref{lem:VforT(2,k)}, we know that 
\[\nu^+(T(2, 2k+1)) = \begin{cases}
k & \text{ if $k\geq 0$} \\
0 & \text{ if $k<0$ }.
\end{cases}\]
In the proof of Theorem \ref{thm:v+ bound}, it is shown that if $K$ can be unknotted with a negative twist of positive linking number $l$, then 
\[ \frac{ 1 + \sqrt{1 + 8{\nu^+(K)}}}{2} \le l <  \frac{ 3 + \sqrt{9 + 8{\nu^+(K)}}}{2}.  \]
In particular, this implies that if $\nu^+(K)\geq 7$, then $l>4$.  This contradicts Proposition \ref{prop:alt} and so we conclude that $k\leq 6$.  Similarly, by considering $\nu^+(-K)$, we conclude that $k\geq -6$ for $K$ to be unknotted with a positive twist of positive linking number.  Applying this bound for each value of $k$  yields a short list of possible values for both positive and negative twists.  Factoring in the signature function obstructions from Section \ref{sec:cg} (and recalling that, with our conventions, $\sigma(K) = \sigma_{1/2}(K)$), further restricts the lists.  In particular, if $\sigma(K)>2$ or $\sigma(K)<-2$, then $0^\pm \not \in \calu(K)$.  We are left with the following possibilities:
\begin{itemize}
\item If $k=0$, then $\calu(K) \subseteq \{2^-, 1^-, 0^-, 0^+, 1^+, 2^+\}$.
\item If $k=\mp1$, then $\calu(K) \subseteq \{ 1^\mp,  0^\mp, 2^\pm, 3^\pm\}$.
\item If $k=\mp2$, then $\calu(K) \subseteq \{ 1^\mp, 3^\pm\}$.
\item If $k=\mp3$, then $\calu(K) \subseteq \{ 1^\mp,  3^\pm, 4^\pm\}$.
\item If $k=\mp4$, then $\calu(K) \subseteq \{ 1^\mp,  4^\pm\}$.
\item If $k=\mp5$, then $\calu(K) \subseteq \{ 1^\mp, 4^\pm\}$.
\item If $k=\mp6$, then $\calu(K) \subseteq \{ 1^\mp, 4^\pm\}$.
\item If $k<-6$, then $\calu(K) \subset \{1^-\}$.
\item If $k>6$, then $\calu(K) \subset \{1^+\}$.
\end{itemize}
Finally, combining Theorem \ref{thm:vals} with Lemma \ref{lem:VforT(2,k)}, we can rule out $3^\pm$ for $k=\mp3$ and $4^\pm$ for $k = \mp5, \mp 6$.
\end{proof}

For large signature, Theorem \ref{thm:arf obs} implies a slightly stronger result.

\begin{corollary}
If $K$ is an alternating knot with $| \sigma (K)| >8$ and $\text{Arf}(K) = 1$, then $K$ cannot be unknotted with a single twist.
\end{corollary}


\section{Comments and Questions}\label{sec:questions}

\begin{enumerate}
\item A census of prime knots with up to 8 crossings reveals there are only eight knots such $K$ for which $\calu(K)$ is completely known.    The first  example with unknown values  is $\calu(5_2)$.  This knot has signature $-2$ and Theorem \ref{thm:alt table} implies that $\calu(5_2) \subseteq \{1^+, 0^+, 2^-, 3^-\}$. We have that Arf$(5_2) = 0$ and Theorem \ref{thm:arf obs} implies that $3^-\not\in \calu (5_2)$. Because $5_2$ has unknotting number 1, realized by a positive-to-negative crossing change, we know that $\{0^+, 2^-\} \subseteq \calu(5_2)$.  Thus the only remaining unknown value is $1^+$.

\item Perhaps the most basic open question about which sets occur as unknotting sets is the following:  Does there exist a nontrivial knot $K$ with  $\calu(K) = \{2^-, 1^-, 0^-, 0^+, 1^+, 2^+\}$?

\item Our results are based primarily on knot invariants that are related to four-manifolds in some way.  As of yet, three-manifold techniques have provided little access to solving the  problem of determining $\calu(K)$ for individual knots.  On the other hand, they seem well-suited for addressing more geometric questions, for instance related to primeness, and for working with families of knots; some examples of this are included in~\cite{MR1355072,MR1181171,MR2016983,MR2107139}.

\item As is evident from our work here, the case of linking number one is especially challenging.    This challenge is related to the difficulty of finding invariants related to homology three-spheres, as opposed to rational homology spheres.  We expect that a continued study of this linking number one problem will bring new focus on particular problems related to homology three-spheres.

\item Ohyama's theorem~\cite{MR1297516} states that any knot can be unknotted with two twists.  A closer look at his construction shows that the linking numbers are consecutive integers.   With more care it can be seen that Ohyama's proof yields the following.
\begin{theorem}  For every integer $l \ge 0$ and knot $K$, it is possible to unknot $K$ with  a pair of oppositely signed twists of linking numbers $l$ and $l+1.$
\end{theorem}
\noindent The results concerning signatures presented in this paper can be   generalized to show that if the   $l+1$ in the statement of the theorem is replaced with $l+k$ for any $k>1$, then it is no longer true.   On the other hand, for a fixed pair $(l_1, l_2)$, we are unable to either offer a generalization or find an obstruction.  For instance, the following statement is possibly true:  {\it Every knot $K$ can by unknotted by a pair of oppositely signed twists of linking numbers $3$ and $5$.}  (Here, 3 and 5 could be replaced by any relatively prime pair.)
\item  The problem of determining whether a given knot has unknotting number one has been resolved for all prime knots of 10 or fewer crossings.  There are 27 knots of 11 or fewer crossings, out of a total of 801 knots, for which it is unknown.  Note that saying that a knot $K$ has unknotting number one implies that $0 \in \calu(K)$ (with some choice of sign) but not conversely.    A good but lengthy project is to review the calculations that went into determining the unknotting numbers to identify knots of low crossing number for which the question of whether $0 \in \calu(K)$ is unresolved.

\end{enumerate}


\appendix

\section{Signatures of torus knots}

Figure~\ref{fig:signature} illustrates Litherland's description of the signature function of $T(p,q)$   in the case of $p=5, q=7$.  In a rectangle with vertices $(0,0)$, $(q,0)$, $(0,p)$, and $(q,p)$, line segments are drawn:  one from $(qx, 0)$ to $(0,px)$, and the other from $(qx,p)$ to $(q, px)$.   The signature at $\omega = e^{2\pi i x}$ is given be counting the number of lattice points interior to the two triangular regions ($C_1$ and $C_2$) and subtracting the number of lattice points in the interior of the remaining region.  In the illustration we have $x = 0.6$ and find ${\sigma}_x(T(p,q)) = 3 + 1 - 20 = -16$.  By symmetry, we can focus on the range $0 \le x \le \frac{1}{2}$.

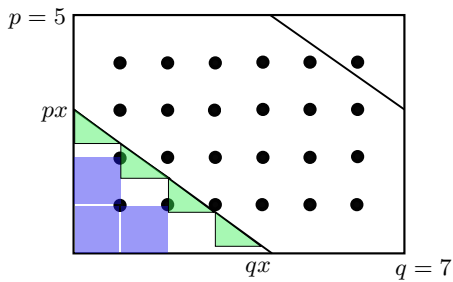
\begin{figure}[ht]
\labellist
\pinlabel {\text{\small{$px$}}} at -35 280
\pinlabel {\text{\small{$p=5$}}} at -70 470
\pinlabel {\text{\small{$qx$}}} at 370 -30
\pinlabel {\text{\small{$q=7$}}} at 700 -30
\endlabellist
\includegraphics[scale=.2]{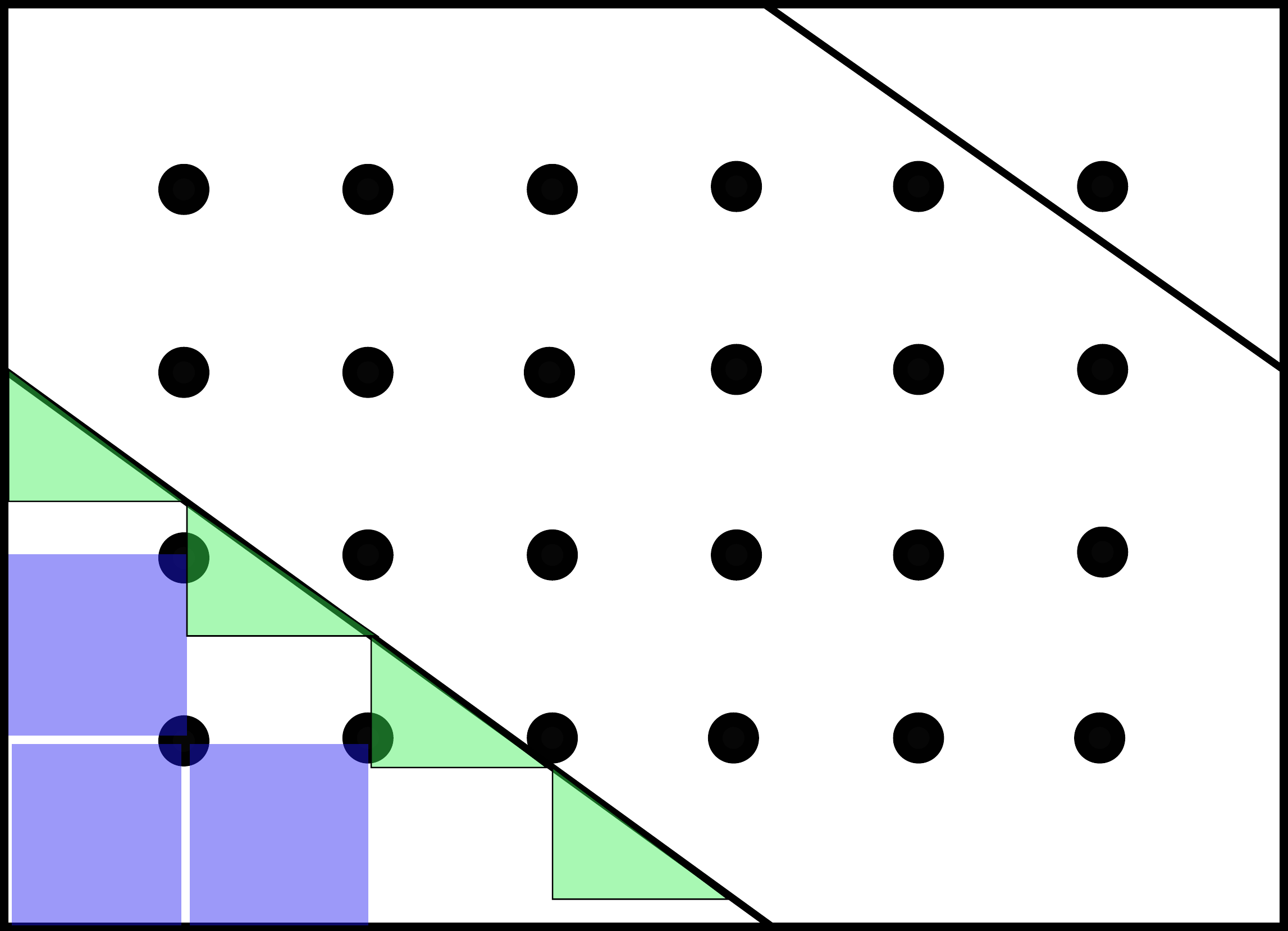}
\vskip.2in

\caption{Signature Count}
\label{fig:signature}
\end{figure}

It will be simpler to work with negative signature, which we denote 
\[
\ns = -  {\sigma}_x(T(p,q)).
\] 
Fix a choice of $p$, $q$, and $x$.  We call the lower and upper triangle counts $\#C_1$ and $\#C_2$, and the remaining count by $\#C_3$, so that $\#C_1 + \#C_2 + \#C_3 = (p-1)(q-1)$.  We then have

\begin{theorem}[\bf Exact Count]

\begin{equation}\label{eqn:sig1}
\ns   =  (p-1)(q-1) - 2(\#C_1 +\#C_2). 
\end{equation}
\end{theorem}

The simplest estimate for the signature function results from a consideration of areas.  The (blue) shaded squares in Figure~\ref{fig:signature} illustrate that the count $\#C_1$ is the sum of the areas of squares, and this sum approximates  the area of the lower triangle.  In general, if we approximate  the counts, $\#C_1$ and $\#C_2$, as well as  the count of the lattice points in the complementary region, by the areas of the regions, we arrive at the following.

\begin{theorem}[\bf Approximation]
\begin{equation}\label{eqn:sigest2}
\ns \approx   pq  -   pqx^2 -   pq (1-x)(1 - x)  =  2pqx(1-x).
\end{equation}
\end{theorem}

We need to improve this to find a precise  lower bound for  the  $\ns$.  To do so, we can subtract the areas of the two triangles from the total number of  lattice points, $(p-1)(q-1)$, yielding the next result.

\begin{theorem}[\bf Lower Approximation] \label{thm:siglower} 
\[
\ns > (p-1)(q-1) -  pqx^2  - pq (1-x)(1 - x)  .
\]
\end{theorem}

\begin{theorem}\label{thm:signsigma} If $0 \le x < \frac{1}{pq}$ then $\ns = 0$.  If  $\frac{1}{pq}< x < \frac{1}{2}$ then $\ns > 0$.  
\end{theorem}

\begin{proof}
Denote the hypotenuses of triangles $C_1$ and $C_2$ by $l_1$ and $l_2$, both of which depend on the choice of $x$.  As $x$ increases from $0$ to $\frac{1}{2}$, the negative signature increases when a lattice point lies on $l_2$ and it decreases when a lattice point lies on $l_1$.

\smallskip

\noindent{\bf Step 1:  The first positive jump in $\ns$.}   The function on the plane $\phi\co (i,j) \to ip + jq$ is constant on each line $l_2$, taking value $pq( 1 + x)$ on the general line and, in particular, taking value $pq$ on the line when $x = 0$.  To find the first positive jump, we must find the smallest $x>0$ such that  the line $l_2$ contains a lattice point    $(i,j)$ in the rectangle  for which  $\phi(i,j) > pq$.  We show that  occurs at  $ x = \frac{1}{pq}$.

We begin by using the fact that $p$ and $q$ are relatively prime:  there exists an $r$ satisfying $0 < r <q$ such that $rp + sq = 1$ for some $s$.  A simple algebraic argument shows that $-p < s <0$.

Consider the point $(i,j) = (r, p+s )$.  Notice that this is in the interior of the rectangle.  We have $\phi(i,j) = rp + (p+s)q = pq +1$.  This is clearly the least possible integer value of $ip+jq$ that is greater than $pq$.     Writing $pq +1 = pq(1 + \frac{1}{pq})$ shows that the corresponding value of $x $ is $\frac{1}{pq}$, as desired.

\smallskip

\noindent{\bf Step 2: The first negative   jump in $\ns$.}  This case is simpler.  It is evident that the first negative jump corresponds the value of $x$ for which the line $l_1$ contains the lattice point $(1,1)$.  It is a simple algebra exercise to show that the value of  $x$ is $\frac{1}{p} + \frac{1}{q}$.

\smallskip

\noindent{\bf Step 3:  $\ns > 0 $ for $x \geq \frac{1}{p} + \frac{1}{q}$.}  The proof of the theorem is completed by showing that  $\ns  > 0 $ for all $x$ satisfying $  \frac{1}{p} + \frac{1}{q} \leq x < \frac{1}{2}$.  The lower bound given in Lemma~\ref{thm:siglower} is quadratic, increasing on $[0,\frac{1}{2}]$.  If we denote that lower bound by $\beta_{p,q}(x)$, we need to check that $\delta = \beta_{p,q}( \frac{1}{p} + \frac{1}{q}) >0$.  A direct substitution and simplification yields 
\[\delta = p + q - 2(\frac{p}{q} + \frac{q}{p}) - 3.\]

If we assume that $p\ge 3$ and $q \ge 5$, then we have 
\[\delta \ge p + q - 2(\frac{p}{5} + \frac{q}{3}) - 3 = \frac{3}{5}p + \frac{1}{3} q -3 \ge 
\frac{9}{5}  + \frac{5}{3}   -3 = \frac{7}{25}.\]
The remaining cases $T(2,k)$ and $T(3,4)$ can be computed explicitly.
\end{proof}

\section{Basic definitions related to Upsilon, $\Upsilon(K)$ and proof of Proposition~\ref{prop:upsilonVbound}}\label{app:upsilon}

Let $K$ be a knot and let $C = \cfk^-(K)$ be the Heegaard Floer knot complex for $K$.  The invariant $V_s$ can be defined to be
\[V_s(K) := -\frac{1}{2}\max\left\{ \text{gr}(x) \mid x\in H_*(C\{i\leq 0, j\leq s\}) \text{ and } U^kx\neq 0 \in H_*(C) \text{ for all } k  \right\}\]
where $gr(x)$ is the Maslov grading of $x$ and $C\{i\leq 0, j\leq s\}$ is the subcomplex of $C$ consisting of elements of Alexander filtration at most $s$ and algebraic filtration at most $0$.  This definition is equivalent to that given in \cite{MR3393360}.

In \cite{MR3667589}, Ozsv\'ath, Stipsicz, and Szab\'o define the knot invariant Upsilon $\U_K(t)$ for $t\in[0,2]$.  Let Alex($x$) denote the Alexander grading of an element in  $ \cfk^-(K)$ and let  Alg($x$) denote the algebraic grading.   Suppose that    $\mathcal{B}$ be a bifiltered basis of $\cfk^\infty (K)$.  In \cite{MR3604374}, it is shown that
\[\U_K(t) = -2 \cdot \min \{r \mid H_0(\mathcal{F}_{t,r}) \longrightarrow H_0(\cfk^\infty (K)) \text{ is surjective}\}\]
where 
$ \mathcal{F}_{t,r} $ is the subcomplex generated by the set 
\[ \left\{x\in \mathcal{B} \left| \left( \frac{t}{2} \text{Alex}(x) +\left(1-\frac{t}{2}\right) \text{Alg}(x)\right)\leq r \right. \right\}.\]
Diagrammatically, the subcomplex $\mathcal{F}_{t,r}$ is represented as a half-space with boundary line 
\[ \frac{t}{2}j+\left(1-\frac{t}{2}\right)i= r.\]  
Note that sums of elements in this half-space are in  $\mathcal{F}_{t,r}$, but might not have bifiltration levels satisfying the given constraint.  
\subsection{Relating $\Upsilon(K)$ to $V_i(K)$}

\begin{proposition} \label{prop:upsilon V bound}
Let $K$ be a knot and $g=g(K)$ be the genus of $K$.  Then for $t\in[0,2]$ and $s\geq 0$, 
\[-st-2V_s(K)\leq \U_t(K)\leq 
\begin{cases}
-gt-2V_s-2s+2g+2 & t\leq 1-\frac{s}{g} \\
gt -2V_s+2 & t\geq 1-\frac{s}{g} 
\end{cases} \]
\end{proposition}

\begin{proof}
Fix $s\geq 0$.  The maximum grading of a generator of homology in $C\{i\leq 0, j\leq s\}$ is $-2V_s$.  Thus the maximum grading of a generator in $C\{i\leq V_s-1, j\leq s+V_s-1\}$ is $-2$ and there is a generator of grading $0$ in $C\{i\leq V_s, j\leq s+V_s\}$.

In particular, if the complex $C\{i\leq V_s, j\leq s+V_s\}$ contains a grading $0$ generator   and the value of $r$ is 
\[r = \left( \frac{t}{2}\left(s+V_s\right)+\left(1-\frac{t}{2}\right)V_s\right),\]
then $\mathcal{F}_{t,r}$  contains a generator of grading $0$ and the map 
$H_0(\mathcal{F}_{t,r}) \longrightarrow H_0(\cfk^\infty (K))$ is surjective.  Thus,
\[\U_K(t)\geq -2\left( \frac{t}{2}\left(s+V_s\right)+\left(1-\frac{t}{2}\right)V_s\right) = -2V_s-ts.\]

On the other hand, since $C\{i\leq V_s-1, j\leq s+V_s-1\}$ does not contain a generator of grading $0$ (the maximum grading here is $-2$), for each $t\in[0,2]$, the minimum $r$--value in the definition of $\U_K(t)$ is such that the following system of inequalities has a (nonempty) solution
\begin{equation}\label{eq:ineq system}\begin{cases}
\frac{t}{2}j+\left(1-\frac{t}{2}\right)i\leq r, & (\text{A})\\
-g\leq j-i\leq g, & (\text{B}) \\
i>V_s-1 \text{ or } j>s+V_s-1, & (\text{C})\\
\end{cases}\end{equation}
where $g = g(K)$ is the genus of the knot $K$.  Combining inequalities \eqref{eq:ineq system}(B) and \eqref{eq:ineq system}(C), we have that if $i>V_s-1$, then
\begin{equation*}
\frac{t}{2}j+\left(1-\frac{t}{2}\right)i = \frac{t}{2}(j-i) + i > -\frac{t}{2}g +V_s-1, 
\end{equation*}
and if $j>s+V_s-1$, then 
\begin{equation*}
\frac{t}{2}j+\left(1-\frac{t}{2}\right)i= \left(\frac{t}{2}-1\right)(j-i)+j > \left(\frac{t}{2}-1\right)g +s+V_s-1.
\end{equation*}
Therefore, if the system of inequalities is to have a solution, either 
\begin{equation}\label{eq:bound 1} r> -\frac{t}{2}g +V_s-1 \end{equation}
or
\begin{equation}\label{eq:bound 2} r>\left(\frac{t}{2}-1\right)g +s+V_s-1. \end{equation}
This implies that for all $t\in[0,2]$, 
\[r>\min\left\{-\frac{t}{2}g +V_s-1, \left(\frac{t}{2}-1\right)g +s+V_s-1\right\}.\]
The two lower bounds agree when $t = 1-\frac{s}{g}$.
When $t\geq 1-\frac{s}{g}$, Inequality \eqref{eq:bound 1} gives the weaker lower bound on $r$ and we have
\[\U_K(t) = -2r < tg-2V_s+2.\]
When $t\leq 1-\frac{s}{g}$, Inequality \eqref{eq:bound 2} gives a weaker lower bound on $r$ and we have
\[\U_K(t) = -2r < -tg+2g-2s-2V_s+2. \qedhere\]
\end{proof}


\section{Results for knots with up to eight crossings}
In Figure \ref{fig:table}, we summarize our result for knots with up to eight crossings.   We find eight knots for which there are no remaining unknown values and 14 knots for which there are no known values.  Candidates for knots with six unknotting indices are $8_3$, $8_9$, and $8_{20}$.

\begin{figure}[h]
\begin{center}
\begin{tabular}{ |c|c|c| } 
 \hline
Knot & Known values & Unknown values\\
 \hline
 $3_1$ & $3^-, 2^-, 0^+$ & \\
 \hline
 $4_1$ & $2^-, 0^-, 0^+, 2^+$ & \\
 \hline
 $5_1$ & $3^-$ & \\
 \hline
 $5_2$ & $2^-, 0^+$ & $1^+$\\
 \hline
 $6_1$ & $2^+, 0^-$ & $2^-, 1^-, 1^+$\\
 \hline
 $6_2$ & $2^-, 0^+$ & $3^-$\\
 \hline
 $6_3$ & $2^-, 0^-, 0^+, 2^+$ & \\
 \hline
 $7_1$ & $4^-$ & \\
 \hline
 $7_2$ & $2^-, 0^+$ & $3^-$ \\
 \hline
 $7_3$ &  & $3^-$ \\
 \hline
 $7_4$ & & $2^-, 0^+, 1^+$\\
 \hline
 $7_5$ & & $1^+$ \\
 \hline
 $7_6$ & $2^-, 0^+$ & $3^-$ \\
 \hline
 $7_7$ & $2^-, 0^+$ & $2^+$ \\
 \hline
 $8_1$ & $0^-, 2^+$ & $2^-, 0^+$ \\
 \hline
 $8_2$ &  & $1^+$ \\
 \hline
 $8_3$ & & $2^-, 1^-, 0^-, 0^+, 1^+, 2^+$ \\
 \hline
 $8_4$ & & $0^-, 2^+, 3^+$\\
 \hline
 $8_5$ &  & $3^-$ \\
 \hline
 $8_6$ &  & $2^-, 0^+, 1^+$\\
 \hline
 $8_7$ & $0^-, 2^+$ & $3^+$ \\
 \hline
 $8_8$ & & $2^-, 1^-, 1^+, 2^+$  \\
 \hline
 $8_9$ &  $2^-, 0^-, 0^+, 2^+$ & $1^-, 1^+$ \\
 \hline
 $8_{10}$ &  & $0^-, 2^+, 3^+$ \\
 \hline
 $8_{11}$ & $2^-, 0^+$ & $3^-$ \\
 \hline
 $8_{12}$ & & $2^-, 0^-, 0^+, 2^+$ \\
 \hline
 $8_{13}$ & $2^-, 0^-, 0^+, 2^+$ & \\
 \hline
 $8_{14}$ & $2^-, 0^+$ & $1^+$ \\
 \hline
 $8_{15}$ & & $1^+$\\
 \hline
 $8_{16}$ & & $0^-, 2^+, 3^+$ \\
 \hline
 $8_{17}$ & $2^-, 0^-, 0^+, 2^+$ & \\
 \hline
 $8_{18}$ & & $2^-, 2^+$ \\
 \hline
 $8_{19}$ & $4^-, 3^-$ &  \\
 \hline
 $8_{20}$ & $0^-, 2^+$ & $2^-, 1^-, 0^+, 1^+$ \\
 \hline
 $8_{21}$ & $2^-, 0^+$ & $1^+$ \\
 \hline
\end{tabular}
\end{center}
\caption{Known and unknown values for knot with up to 8 crossings.  Here, ``known values'' are those which are confirmed to be in the set of unknotting indices for the given knot.  The ``unknown values'' are those which cannot yet be ruled out, but for which realizability is unknown.}
\label{fig:table}
\end{figure}

 Most of the known values are a result of the observation that if $K$ has unknotting number 1, then either $\{0^-, 2^+\} \subset U(K)$ or $\{0^+, 2^-\} \subset U(K)$, depending on the sign of the crossing change needed to unknot $K$.
For the two knots $3_1$ and $8_{19}$ (the torus knots $T(2,3)$ and $T(3,4)$, respectively), knowledge of their diagrams contributes to the known values.  That the knots $5_1$ and $7_1$ can be unknotted with negative twists of linking numbers 3 and 4, respectively, is shown in \cite[Figure 6]{MR2569563} and \cite[Figure 5]{MR2569563}, respectively.  Finally, in \cite{MR2569563}, Ait Nouh shows that the knot $7_1$ is not slice in $\CP^2$.  This rules out all remaining values for $7_1$.



\bibliography{BibTex-UT}
\bibliographystyle{plain}	

\end{document}